\documentclass{amsart}
\usepackage{tikz}
\usepackage{subcaption}
\usetikzlibrary{decorations.pathreplacing}
\usetikzlibrary{decorations.pathmorphing}
\usepackage{stmaryrd}
\usepackage{url}
\usepackage{xcolor}
\usepackage{enumerate}
\usepackage{amssymb}
\usepackage{amsthm}
\usepackage{amsmath}
\usepackage{color}

\renewcommand{\supset}{\supseteq}

\newcommand{\fmax}{f_{\mathord{\max}}}

\tikzstyle{stretchy}=[decorate, decoration={snake, segment length=.2cm, amplitude=.05cm}]
\tikzstyle{arbpath}=[densely dotted]

\newcommand{\st}{\colon\,}
\tikzstyle{vertex}=[inner sep = 0pt, minimum width=4pt, fill=black, shape=circle]
\tikzstyle{squarevert}=[inner sep = 0pt, minimum width=4pt, minimum height=4pt, fill=white, shape=rectangle, draw=black, thick]
\tikzstyle{trivert}=[inner sep = 0pt, minimum width=6pt, minimum height=6pt, fill=gray!50!white, shape=regular polygon,regular polygon sides=3, draw=black, thick]
\tikzstyle{pentvert}=[inner sep = 0pt, minimum width=6pt, minimum height=6pt, fill=gray!15!white, shape=regular polygon,regular polygon sides=5, draw=black, thick]

\newcommand{\after}{\mathop{\circ}}
\newcommand{\iso}{\cong}

\newcommand{\anonlist}[2]{\node[style=vertex, label=#1:\vphantom{pqrsxyz}#2]}
\newcommand{\gpoint}[2]{\node[style=vertex, label=#1:$#2$]}

\newcommand{\apoint}[1]{\gpoint{above}{#1}}
\newcommand{\lpoint}[1]{\gpoint{left}{#1}}
\newcommand{\rpoint}[1]{\gpoint{right}{#1}}

\renewcommand{\subset}{\subseteq}

\DeclareMathOperator{\dam}{dam}

\DeclareMathOperator{\ch}{ch}

\newtheorem{proposition}{Proposition}[section]
\newtheorem{lemma}[proposition]{Lemma}
\newtheorem{theorem}[proposition]{Theorem}
\newtheorem{corollary}[proposition]{Corollary}
\newtheorem{conjecture}[proposition]{Conjecture}
\theoremstyle{definition}
\newtheorem{definition}[proposition]{Definition}
\newtheorem{observation}[proposition]{Observation}
\newtheorem{example}[proposition]{Example}

\newcommand{\bigsizeof}[1]{\left\lvert{#1}\right\rvert}
\newcommand{\sizeof}[1]{\lvert{#1}\rvert}

\title{On $\ab{4}{2}$-choosable Graphs}
\date{\today}

\newcommand{\etaketwo}[3]{#1 \ominus (#2, #3)}%{#1 \ominus \llbracket #2, #3 \rrbracket}
\newcommand{\Ltaketwo}[2]{\etaketwo{L}{#1}{#2}}

\newcommand{\ab}[2]{(#1, #2)}

\author{Jixian Meng}
\author{Gregory J.~Puleo}
\author{Xuding Zhu}
\address[Jixian Meng, Xuding Zhu]{Department of Mathematics, Zhejiang Normal University}
\address[Gregory J.~Puleo]{Department of Mathematics, University of Illinois at Urbana--Champaign}
\thanks{This research supported partially by grant CNSF 11571319.}

\begin{document}
\maketitle
\begin{abstract}
  A graph $G$ is called $\ab{a}{b}$-choosable if for any list
  assignment $L$ which assigns to each vertex $v$ a set $L(v)$ of $a$
  permissible colours, there is a $b$-tuple $L$-colouring of $G$. An
  $\ab{a}{1}$-choosable graph is also called $a$-choosable. In the
  pioneering paper on list colouring of graphs by Erd\H{o}s, Rubin and
  Taylor~\cite{rubin}, $2$-choosable graphs are
  characterized. Confirming a special case of a conjecture
  in~\cite{rubin}, Tuza and Voigt~\cite{tuzavoigt-twom} proved that
  $2$-choosable graphs are $\ab{2m}{m}$-choosable for any positive
  integer $m$.  On the other hand, Voigt \cite{voigt} proved that if
  $m$ is an odd integer, then these are the only
  $\ab{2m}{m}$-choosable graphs; however, when $m$ is even, there are
  $\ab{2m}{m}$-choosable graphs that are not $2$-choosable. A graph is
  called $3$-choosable-critical if it is not $2$-choosable, but all
  its proper subgraphs are $2$-choosable. Voigt conjectured that for
  every positive integer $m$, all bipartite $3$-choosable-critical
  graphs are $\ab{4m}{2m}$-choosable. In this paper, we determine
  which $3$-choosable-critical graphs are $\ab{4}{2}$-choosable,
  refuting Voigt's conjecture in the process.  Nevertheless, a weaker
  version of the conjecture is true: we prove that there is an even
  integer $k$ such that for any positive integer $m$, every bipartite
  $3$-choosable-critical graph is $\ab{2km}{km}$-choosable. Moving
  beyond $3$-choosable-critical graphs, we present an infinite family
  of non-$3$-choosable-critical graphs which have been shown by
  computer analysis to be $\ab{4}{2}$-choosable. This shows that the
  family of all $\ab{4}{2}$-choosable graphs has rich structure.
\end{abstract}

\section{Introduction}

Multiple list colouring of graphs was introduced in the 1970s by
Erd\H{o}s, Rubin and Taylor~\cite{rubin}.  A \emph{list assignment} is
a function $L$ which assigns to each vertex $v$ a set of permissible
colours $L(v)$.  A \emph{$b$-tuple colouring} of a graph $G$ is a
function $f$ that assigns to each vertex $v$ a set $f(v)$ of $b$
colours so that $f(u) \cap f(v)= \emptyset$ for any edge $uv$ of $G$.
Given a list assignment $L$ of $G$, a \emph{$b$-tuple $L$-colouring}
of $G$, also called an $\ab{L}{b}$-colouring of $G$, is a $b$-tuple
colouring $f$ of $G$ with $f(v) \subseteq L(v)$ for all $v \in V(G)$.
We say $G$ is \emph{$\ab{L}{b}$-colourable} if there is a $b$-tuple
$L$-colouring of $G$, and say $G$ is \emph{$\ab{a}{b}$-choosable} if
$G$ is $\ab{L}{b}$-colourable for any list assignment $L$ with
$\sizeof{L(v)} = a$ for all $v$.  A $\ab{a}{1}$-choosable graph is
also called \emph{$a$-choosable}.  The \emph{choice number} $\ch(G)$
of a graph $G$ is the smallest integer $a$ such that $G$ is
$a$-choosable. List colouring of graphs has been studied extensively
in the literature; see \cite{tuzasurvey} for a survey.

The family of $2$-choosable graphs was characterized by Erd\H{o}s,
Rubin and Taylor~\cite{rubin}.  These graphs have very simple
structure. We define the \emph{core} of a graph $G$ to be the graph
obtained by iteratively deleting vertices of degree $1$. It is easy to
see that a graph is $2$-choosable if and only if its core is
$2$-choosable.  It was proved in \cite{rubin} that a graph $G$ is
$2$-choosable if and only if its core is $K_1$ or an even cycle or
$\Theta_{2,2,2p}$ for some positive integer $p$, where $\Theta_{r,s,t}$ is
the graph consisting of two end vertices $u$ and $v$ joined by three
internally vertex-disjoint paths containing $r$, $s$, and $t$ edges
respectively. Erd\H os, Rubin, and Taylor~\cite{rubin} conjectured
that if a graph is $\ab{a}{b}$-choosable, then it is
$\ab{am}{bm}$-choosable for every positive integer $m$;
Tuza and Voigt~\cite{tuzavoigt-twom} confirmed a special case of this conjecture
by proving that all $2$-choosable graphs are $\ab{2m}{m}$-choosable
for all $m$, but the conjecture is otherwise open. Moreover,
Voigt~\cite{voigt} proved that if $m$ is an odd integer, then these
are the only $\ab{2m}{m}$-choosable graphs.
% \begin{theorem}
% \label{odd}
% \cite{voigt}
%   For any odd integer $m$, a graph $G$ is $\ab{2m}{m}$-choosable  if and only if
%   its core is $K_1$ or an even cycle or $\theta_{2,2,2p}$ for some positive integer $p$.
% \end{theorem}

When $m$ is even, the family of $\ab{2m}{m}$-choosable graphs has much
richer structure.  A \emph{$b$-tuple $a$-colouring} of a graph $G$ is
a $b$-tuple colouring $f$ of $G$ with $f(v) \subseteq \{1,2,\ldots,
a\}$ for each $v$. We say $G$ is \emph{$\ab{a}{b}$-colourable} if such
a colouring exists.  Alon, Tuza, and Voigt~\cite{atv} showed that if a
graph $G$ is $\ab{a}{b}$-colourable, then there is a positive integer
$k_G$ such that $G$ is $\ab{ak_Gm}{bk_Gm}$-choosable for all $m$. In
particular, for any bipartite graph $G$, there is a positive integer
$m_G$ such that $G$ is $\ab{2m_G}{m_G}$-choosable.

This paper is devoted to the study of $\ab{4m}{2m}$-choosability. In
particular, we are interested in the question of which graphs are
$\ab{4}{2}$-choosable.

A graph $G$ is called \emph{$3$-choosable-critical} if $G$ is not
$2$-choosable, but any proper subgraph is $2$-choosable.  The family
of $3$-choosable-critical graphs is characterized by
Voigt~\cite{voigt}:
\begin{theorem}[Voigt~\cite{voigt}]
  A graph is $3$-choosable-critical if and only if
  it is one of the following:
  \begin{enumerate}[(a)]
  \item two vertex-disjoint even cycles joined by a path,
  \item two even cycles with exactly one vertex in common,
  \item a $\Theta_{2r, 2s, 2t}$-graph or $\Theta_{2r-1, 2s-1,
      2t-1}$-graph with $r \geq 1$ and $s,t > 1$,
  % \item a $\Theta_{2r, 2s, 2t}$-graph with $r \geq 1$ and
  %   $s,t > 1$ or $\Theta_{2r+1, 2s+1, 2t+1}$-graph with
  %   $r \geq 0$ and $s,t \geq 1$,
  \item a $\Theta_{2,2,2,2t}$-graph with $t \geq 1$,
  \item an odd cycle.
  \end{enumerate}
\end{theorem}
Voigt conjectured that for any positive integer $m$, all bipartite
$3$-choosable-critical graphs are $\ab{4m}{2m}$-choosable.  In this
paper, we prove the following characterization of the
$\ab{4}{2}$-choosable $3$-choosable-critical graphs, which refutes
Voigt's conjecture:
\begin{theorem}\label{thm:main}
  A $3$-choosable-critical graph is $\ab{4}{2}$-choosable
  if and only if it is one of the following:
  \begin{enumerate}[(a)]
  \item two vertex-disjoint even cycles joined by a path,
  \item two even cycles with exactly one vertex in common,
  \item a $\Theta_{2,2s,2t}$-graph or $\Theta_{1, 2s-1, 2t-1}$-graph with $s,t > 1$,
  \item $\Theta_{2,2,2,2}$.
  \end{enumerate}
\end{theorem}
(Note that $\Theta_{2,2,2,2} \iso K_{2,4}$.)  In particular, among the
bipartite $3$-choosable-critical graphs, when $r,s,t$ have the same
parity and $\min\{r,s,t\} \geq 3$, the graph $\Theta_{r,s,t}$ fails to
be $\ab{4}{2}$-choosable, and when $t > 1$, the graph
$\Theta_{2,2,2,2t}$ fails to be $\ab{4}{2}$-choosable.

Nevertheless, a weaker version of Voigt's conjecture is true:
\begin{theorem}\label{thm:main2}
  There is a fixed integer $k$ such that for every positive integer $m$, every
  bipartite $3$-choosable-critical graph is $\ab{4km}{2km}$-choosable.
\end{theorem}

The paper is structured as follows. In Section~\ref{sec:damage} we
introduce the main lemmas and definitions needed for the proof of
Theorem~\ref{thm:main}. In Section~\ref{sec:lemmas} we collect some
more useful lemmas of a more technical nature.

In Section~\ref{sec:goodtheta} we prove that theta graphs of the form
$\Theta_{2,2s,2t}$ and $\Theta_{1, 2s-1, 2t-1}$ are
$\ab{4}{2}$-choosable. In Section~\ref{sec:evencycle} we apply these
results to show that if $G$ consists of two vertex-disjoint even
cycles joined by a path or two even cycles sharing a vertex, then $G$
is $\ab{4}{2}$-choosable. Tuza and Voigt have already
shown~\cite{tuzavoigt-k24} that $K_{2,4}$ is $\ab{4m}{2m}$-choosable
for all $m$, so this completes the positive direction of
Theorem~\ref{thm:main}.

In Section~\ref{sec:ctx} we present list assignments showing that
$\Theta_{3,3,3}$, $\Theta_{4,4,4}$, and $\Theta_{2,2,2,4}$ are not
$\ab{4}{2}$-choosable; a quick argument given in that section shows
that the larger theta graphs also fail to be
$\ab{4}{2}$-choosable. This completes the characterization of the
$\ab{4}{2}$-choosable $3$-choosable-critical graphs.

In Section~\ref{sec:voigt}, we prove Theorem \ref{thm:main2}.  In
Section~\ref{sec:all42}, we present some non-$3$-choosable-critical
graphs and briefly discuss the computer analysis that demonstrates that these
graphs are $\ab{4}{2}$-choosable. We close with a conjectured
characterization of the $\ab{4}{2}$-choosable graphs.

\section{Paths and Damage}\label{sec:damage}
\begin{figure}
  \centering
  \begin{tabular}{r|c|c|c}
    &\tikz[remember picture]{\apoint{v_1} (v1) at (0cm, 0cm) {};}&
    \tikz[remember picture]{\apoint{v_2} (v2) at (0cm, 0cm) {};}&
    \tikz[remember picture]{\apoint{v_3} (v3) at (0cm, 0cm) {};}\\\hline
    $L(v_i)$& abcd & abef & adeg \\\hline
    $X_i$ & abcd & ef & adg \\\hline
    $\hat{X}_i$ & cd && dg
  \end{tabular}
  \tikz[remember picture, overlay]{\draw (v1) -- (v2) -- (v3);}
  \caption{Example computations of $X_i$ and $\hat{X}_i$. Here $A = \{\mathrm{a}\}$.}
  \label{fig:xcomp-example}
\end{figure}
\begin{definition}
  When $P$ is an $n$-vertex path with vertices $v_1, \ldots, v_n$
  in order, and $L$ is a list assignment on $P$, we define
  sets $X_1, \ldots, X_n$ by
  \begin{align*}
    X_1 &= L(v_1), \\
    %XD X_i &= L(v_i) - X_i \qquad (i \in \{2, \ldots, n\}).
    X_i &= L(v_i) - X_{i-1} \qquad (i \in \{2, \ldots, n\}).
  \end{align*}
  We also define the quantity $S_L(P)$ by
  \[ S_L(P) = \sum_{i=1}^n\sizeof{X_i}. \]
\end{definition}
\begin{lemma}\label{lem:xuding}
  Let $P$ be an $n$-vertex path and let $L$ be a list assignment on
  $P$ such that $\sizeof{L(v_1)}, \sizeof{L(v_n)} \geq 2m$ and
  $\sizeof{L(v_i)} = 4m$ for $i \in \{2, \ldots, n-1\}$. The path $P$
  is $\ab{L}{2m}$-colourable if and only if $S_L(P) \geq 2mn$.
\end{lemma}
\begin{proof}
  We use induction on $n$. The claim is trivial for $n=1$.  Assume that $n
  \ge 2$ and the claim holds for $n' < n$. Let $P' = P - v_n$, and
  observe that if $X'_1, \ldots X'_{n-1}$ are computed as above for
  $P'$, then $X'_i = X_i$ for all $i \in \{1, \ldots, n-1\}$.  Since
  $\sizeof{X_1} \geq 2m$ and $\sizeof{X_i} + \sizeof{X_{i-1}} \geq
  \sizeof{L(v_i)} \geq 4m$ for $i \in \{2, \ldots, n-1\}$, we have
  $\sum_{i=1}^{n-1}\sizeof{X'_i} \geq 2(n-1)m$.

  First assume $S_L(P) \geq 2nm$. We shall prove that $P$ is
  $\ab{L}{2m}$-colourable.  We determine a $2m$-set of colours
  $\phi(v_n)$ to be assigned to $v_n$ as follows: when $\sizeof{X_n}
  \geq 2m$, let $\phi(v_n)$ be any $2m$-subset contained in $X_n$;
  when $\sizeof{X_n} < 2m$, let $\phi(v_n)$ be any $2m$-subset of
  $L(v_n)$ containing $X_n$.

  Let $L^*$ be the restriction of $L$ to $P'$, except that
  $L^*(v_{n-1})=L(v_{n-1})- \phi(v_n)$. When $\sizeof{X_n} \geq 2m$,
  we have $S_{L^*}(P') = \sum_{i=1}^{n-1}\sizeof{X_i} \ge 2(n-1)m$,
  since $\phi(v_n) \cap X_{n-1} = \emptyset$; when $\sizeof{X_n} <
  2m$, we have $S_{L^*}(P') \geq
  \sum_{i=1}^{n}\sizeof{X_i}-\sizeof{\phi(v_n)} \geq 2(n-1)m$, since
  $\phi(v_n) \supset X_n$.  Either way, by the induction hypothesis,
  $P'$ has an $\ab{L^*}{2m}$-colouring, which extends to an
  $\ab{L}{2m}$-colouring of $P$ by assigning $\phi(v_n)$ to $v_n$.

  For the other direction, let $\phi$ be an $\ab{L}{2m}$-colouring of
  $P$.  Let $L^*$ be the restriction of $L$ to $P'$, except that
  $L^*(v_{n-1})=L(v_{n-1})- \phi(v_n)$, and let $X^*_1, \ldots X^*_{n-1}$
  be computed for $L^*$. Since $\phi$ is an $\ab{L^*}{2m}$-colouring
  of $P'$, the induction hypothesis implies
  that \[\sum_{i=1}^{n-1}\sizeof{X^*_i} = S_{L^*}(P') \ge 2(n-1)m.\] It is
  easy to verify that $X_i=X^*_i$ for $i=1,2,\ldots, n-2$, and
  $\sizeof{X_{n-1}} + \sizeof{X_n} \ge
  \sizeof{X^*_{n-1}}+\sizeof{\phi(v_n)} \ge \sizeof{X^*_{n-1}} + 2m$.
  Hence $S_L(P) \ge 2nm$.
\end{proof}
Our typical strategy for showing that a graph $G$ is
$\ab{4m}{2m}$-choosable is as follows: identify a set of vertices $X$
such that $G-X$ is a linear forest (disjoint union of paths), and find
a precolouring of $X$ such that each path $P$ in $G-X$ satisfies
$S_{L^*}(P) \geq 2m\sizeof{V(P)}$, where $L^*$ is obtained from $L$ by
removing from each vertex of $G-X$ the colours used on its neighbors
in $X$. Provided that the degree-$2$ vertices of $G-X$ have no
neighbors in $X$, Lemma~\ref{lem:xuding} then guarantees that we can
extend the precolouring of $X$ to the rest of the graph, as desired.

In order to carry out this strategy, we need to know how $S_L(P)$
changes when colours are removed from the endpoints of $P$. We will be
particularly interested in the case where $P$ has an odd number of
vertices. Before stating the results, we set up some more notation.
\begin{definition}
  If $L$ is a list assignment on an $n$-vertex path $P$ and $S,T$ are
  sets of colours, we define $\Ltaketwo{S}{T}$ to be the list
  assignment obtained from $L$ by deleting all colours in $S$ from
  $L(v_1)$, all colours in $T$ from $L(v_n)$, and leaving all other
  lists unchanged.
\end{definition}
\begin{definition}
  Let $L$ be a list assignment on an $n$-vertex path $P$, where $n$ is odd. Define \[A=\bigcap_{x \in
    V(P)}L(x).\]
  Let
  \begin{align*}
    \hat{X}_1 &= \{c \in L(v_1) - A \st \text{the smallest index $i$ for which $c \notin L(v_i)$ is even}\}. \\
    \hat{X}_n &= \{c \in L(v_n) - A \st \text{the largest index $i$ for which $c \notin L(v_i)$ is even}\}.
  \end{align*}
\end{definition}
See Figure~\ref{fig:xcomp-example} for an example of $\hat{X}_1$ and $\hat{X}_n$.
\begin{observation}
  If $P$ is an $n$-vertex path, where $n$ is odd, then for
  any list assignment on $P$, we have $\hat{X}_n = X_n - A$.
\end{observation}
\begin{lemma}\label{lem:oddhat}
  Let $L$ be a list assignment on an $n$-vertex path $P$, where $n$ is odd.
  For any sets of colours $S,T$, we have
  \[ S_{\Ltaketwo{S}{T}}(P) = S_L(P) - \left(\bigsizeof{(A \cup \hat{X}_1) \cap S} + \bigsizeof{(A \cup \hat{X}_n) \cap T} - \bigsizeof{A \cap S \cap T}\right). \]
\end{lemma}
\begin{figure}
  \centering
  \begin{subfigure}[t]{.4\textwidth}
  \begin{tabular}{r|c|c|c}
    &\tikz[remember picture]{\apoint{v_1} (v1) at (0cm, 0cm) {};}&
    \tikz[remember picture]{\apoint{v_2} (v2) at (0cm, 0cm) {};}&
    \tikz[remember picture]{\apoint{v_3} (v3) at (0cm, 0cm) {};}\\\hline
    $L$& abcd & abef & adeg \\\hline
    $X_i$ & abcd & ef & adg \\\hline
    $\hat{X}_i$ & cd && dg
  \end{tabular}
  \tikz[remember picture, overlay]{\draw (v1) -- (v2) -- (v3);}
  \end{subfigure}
  \begin{subfigure}[t]{.4\textwidth}
  \begin{tabular}{r|c|c|c}
    &\tikz[remember picture]{\apoint{v_1} (v1) at (0cm, 0cm) {};}&
    \tikz[remember picture]{\apoint{v_2} (v2) at (0cm, 0cm) {};}&
    \tikz[remember picture]{\apoint{v_3} (v3) at (0cm, 0cm) {};}\\\hline
    $\Ltaketwo{S}{T}$& cd & abef & de \\\hline
    $X_i$ & cd & abef & d\\
    \multicolumn{4}{c}{\vphantom{$\hat{X}_i$}}
  \end{tabular}
  \tikz[remember picture, overlay]{\draw (v1) -- (v2) -- (v3);}
  \end{subfigure}
  \caption{Example computations for $\Ltaketwo{S}{T}$ when $(S,T) = (\textrm{ab}, \textrm{ag})$.}
  \label{fig:removecolors}
\end{figure}
\begin{proof}
  It suffices to consider the effect of deleting just one colour $c$.
  First we consider deleting the colours in $T$ from $L(v_n)$. Clearly,
  if $c \notin X_n$ then deleting the colour $c$ from $L(v_n)$ has no
  effect on $S_L(P)$, since it does not change any $X_i$. On the other
  hand, if $c \in X_n = A \cup \hat{X}_n$, then deleting the colour $c$
  from $L(v_n)$ decreases $S_L(P)$ by exactly $1$.

  Next we consider deleting a colour $c$ from $L(v_1)$. Here, unlike
  with $L(v_n)$, the changes in $X_1$ can ``ripple'' through later
  $X_i$, as shown in Figure~\ref{fig:removecolors}. If $c \notin X_1 =
  L(v_1)$, then deleting $c$ from $L(v_1)$ clearly does not change any
  $X_i$, hence does not change $S_L(P)$.

  Now suppose $c \in X_1 - A$.  Deleting $c$ from $L(v_1)$ causes $c$
  to be removed from $X_1$. However, if $c \in L(v_2)$,
  we gain $c$ in $X_2$. Now this may cause us to lose $c$
  in $X_3$, gain $c$ in $X_4$, and so forth. The process continues
  until we reach an index $i$ for which $c \notin L(v_i)$.
  If $i$ is odd, then   we lose $c$ from the sets
  $X_1, X_3, \ldots, X_{i-2}$ and gain $c$ in the sets $X_2, X_4,
  \ldots, X_{i-1}$. So there is no net change in $S_L(P)$.
  If $i$ is even, then  we lose $c$ from the sets
  $X_1, X_3, \ldots, X_{i-1}$ and gain $c$ in the sets $X_2, X_4,
  \ldots, X_{i-2}$. So $S_L(P)$ has decreased by $1$.

  Finally, suppose $c \in X_1 \cap A$. Deleting $c$ from $L(v_1)$
  causes the same ripple process described above, terminating when we
  try to delete $c$ from $X_n$ (since $n$ is odd). If $c \notin T$,
  then as before, this causes $S_L(P)$ to decrease by $1$. However, if
  $c \in T$, then we have \emph{already} deleted $c$ from $X_n$, so in
  this step we really gain and lose $c$ an equal number of times.
  Thus, when $c \in A \cap S \cap T$,
  deleting $c$ from both endpoints of $L$ decreases $S_L(P)$ by
  exactly $1$, but such colours are double-counted in the sum
  $\sizeof{(A \cup \hat{X}_1) \cap S} + \sizeof{(A \cup \hat{X_n})
    \cap T}$. The final term $\sizeof{A \cap S \cap T}$ corrects for
  this overcount.
\end{proof}

%%%%%%%%

%%%%%%%%

Together, Lemma~\ref{lem:xuding} and Lemma~\ref{lem:oddhat}  allow us to ignore the details of the list assignment and
focus on the sets $\hat{X}_1, \hat{X}_n, A$, as described below.
\begin{definition}
For a pair of colour sets $S,T$, the \emph{damage} of $(S,T)$ with respect to $L$ and $P$ is written
$\dam_{L,P}(S,T)$ and defined by
\[\dam_{L,P}(S,T)=S_L(P) -   S_{\Ltaketwo{S}{T}}(P).\]
\end{definition}
Lemma \ref{lem:oddhat} shows that if $P$ has an odd number of
vertices, then given a pair $S,T$ of colour sets, the damage
$\dam_{L,P}(S,T)$ just depends on $\hat{X}_1, \hat{X_n}$ and $A$, and
in particular
\begin{align}
  \label{eq:damformula}
\dam_{L,P}(S,T) &= \sizeof{(A \cup \hat{X}_1) \cap S} + \sizeof{(A \cup \hat{X}_n) \cap T} - \sizeof{A \cap S \cap T} \\
&= \sizeof{\hat{X}_1 \cap S} + \sizeof{\hat{X}_n \cap T} + \sizeof{A \cap (S \cup T)}.\nonumber
\end{align}
In the example of Figure~\ref{fig:removecolors}, we have $\dam_{L,P}(S,T) = 2$.
\begin{lemma}\label{lem:abstract}
  Let $G$ be a graph, and let $X \subset V(G)$ be a set of vertices
  such that every component of $G-X$ is a path with an odd number of
  vertices.  Assume that for each component $P$ of $G-X$, only the two
  end vertices of $P$ have neighbors in $X$. Let $L$ be a list
  assignment on $G$ with $\sizeof{L(v)}=4m$ for all $v \in V(G)$. The
  graph $G$ is $\ab{L}{2m}$-colourable if and only if $G[X]$ has an
  $\ab{L}{2m}$-colouring $\phi$ such that for every path $P$ in $G-X$
  with vertices $v_1, \ldots, v_n$ in order, the following conditions hold:
  \begin{enumerate}[(i)]
  \item $\sizeof{L(v_1) \cap \phi(N_X(v_1))} \leq 2m$, \\
  \item $\sizeof{L(v_n) \cap \phi(N_X(v_n))} \leq 2m$, and \\
  \item $\dam_{L,P}(\phi(N_X(v_1)), \phi(N_X(v_n)))
    \leq S_L(P) - 2mn$.
  \end{enumerate}
\end{lemma}
\begin{proof}
  Clearly, $G$ is $\ab{L}{2m}$-colourable if and only if $G[X]$ has an
  $\ab{L}{2m}$-colouring $\phi$ that extends to $G$, i.e., extends to each of the paths
  $P$ in $G-X$. For each path
  $P$ of $G-X$, we show that $\phi$ extends to $P$ if and only if $\phi$
  satisfies conditions (i)--(iii). Conditions (i) and (ii) are clearly
  necessary, so it suffices to show that when Conditions (i) and (ii)
  hold, $\phi$ extends to $P$ if and only if Condition~(iii)
  holds. This follows from Lemma~\ref{lem:xuding} and
  Lemma~\ref{lem:oddhat}.
\end{proof}

\section{Technical Lemmas}\label{sec:lemmas}
To apply Lemma \ref{lem:abstract}, we need to find lower bounds for
$S_L(P)$ and upper bounds for $\dam_{L,P}(S,T)$.  In this section, we
collect some technical lemmas regarding such bounds.
\begin{lemma}\label{lem:extra}
  If $L$ is a list assignment on an $n$-vertex path $P$, where $n$ is odd and $\sizeof{L(v_i)} = 4m$
  for all $i$, then
  \[ S_L(P) = 2nm - 2m + \sum_{\substack{\text{$k$ even}\\k < n}}\sizeof{X_{k-1} - L(v_k)} + \sizeof{X_n}. \]
\end{lemma}
\begin{proof}
  We use induction on $n$. When $n=1$, the sum is empty and $2nm-2m = 0$, so
  the claim is just $S_L(P) = \sizeof{X_1}$, which is clearly
  true. Assume that $n > 1$ and the claim holds for smaller odd $n$.
  Let $P' = P - \{v_{n-1}, v_n\}$ and let $L'$ be the restriction of $L$ to $P'$,
  so that $S_L(P) = S_{L'}(P') + \sizeof{X_{n-1}} + \sizeof{X_n}$. Applying the
  induction hypothesis to $P'$ yields
  \[ S_L(P) = \left(2nm - 6m + \sum_{\substack{\text{$k$ even}\\k < n -2}}\sizeof{X_{k-1} - L(v_k)} + \sizeof{X_{n-2}}\right) + \sizeof{X_{n-1}} + \sizeof{X_n} \]
  Observe that
  \begin{align*}
    \sizeof{X_{n-1}} &= \sizeof{L(v_{n-1}) - X_{n-2}} \\
    &= \sizeof{L(v_{n-1})} - \sizeof{X_{n-2}} + \sizeof{X_{n-2} - L(v_{n-1})} \\
    &= 4m - \sizeof{X_{n-2}} + \sizeof{X_{n-2} - L(v_{n-1})}
  \end{align*}
  Combining these terms with the terms from $S_{L'}(P')$ gives the
  desired expression for $S_L(P)$.
\end{proof}
\begin{lemma}\label{lem:hatsum}
  If $L$ is a list assignment on an $n$-vertex path $P$, where $n$ is odd and $\sizeof{L(v_i)} = 4m$
  for all $i$, then
  \[ S_L(P) \geq 2nm-2m + \sizeof{\hat{X}_1} + \sizeof{\hat{X}_n} + \sizeof{A}.
   %= 2nm-2m + \sizeof{\hat{X}_1 \cap X_n} + \sizeof{\hat{X}_1 \cup X_n}.
  \]
\end{lemma}
\begin{proof}
  By the definition of $\hat{X}_1$, every element of $\hat{X}_1$ appears
  in a set of the form $X_{k-1} - L(v_k)$ where $k$ is even. Thus,
  the claim follows from Lemma~\ref{lem:extra}, since $\sizeof{X_n} = \sizeof{\hat{X}_n} + \sizeof{A}$.
\end{proof}
\begin{lemma}\label{lem:trivbound}
  If $L$ is a list assignment on an $n$-vertex path $P$, where $n$ is
  odd and $\sizeof{L(v_i)} = 4m$ for all $i$, then $S_L(P) \geq 2nm +
  2m$.
\end{lemma}
\begin{proof}
  This follows immediately from the definition $S_L(P) =
  \sum_{i=1}^n\sizeof{X_i}$ and the observations that $\sizeof{X_1} =
  \sizeof{L(v_1)} = 4m$ and that $\sizeof{X_i} + \sizeof{X_{i+1}} \geq
  4m$ for $i > 1$.
\end{proof}
\section{$\ab{4}{2}$-choosable Theta Graphs}\label{sec:goodtheta}
In this section, we show that $\Theta_{r,s,t}$ is
$\ab{4}{2}$-choosable if $r,s,t$ have the same parity and
$\min\{r,s,t\} \leq 2$.  In Section~\ref{sec:ctx}, we will show that
$\min\{r,s,t\} \geq 3$ implies that $\Theta_{r,s,t}$ is not
$\ab{4}{2}$-choosable. As we are only concerned with
$\ab{4}{2}$-choosability, we will tacitly assume that all list
assignments considered in this section have $\sizeof{L(v)} = 4$ for
all $v \in V(G)$.

We first use an observation of Voigt to restrict to the case where $r,s,t$
are even.
\begin{lemma}[Lemma 4.3 of Voigt~\cite{voigt}]\label{lem:rubin}
  Let $G$ be a graph, let $v \in V(G)$, and let $G'$ be obtained
  from $G$ by deleting $v$ and merging its neighbors. If $G$
  is $\ab{4m}{2m}$-choosable, then $G'$ is $\ab{4m}{2m}$-choosable.
\end{lemma}
The transformation used in Lemma~\ref{lem:rubin} was first used in
\cite{rubin}, which observed that if $G$ is $2$-choosable, then $G'$
is also $2$-choosable.  Voigt~\cite{voigt} made the stronger
observation that if $G$ is $\ab{2m}{m}$-choosable, then $G'$ is also
$\ab{2m}{m}$-choosable. While Voigt imposed the additional assumption
that $d(v)=2$, this assumption is not necessary.
\begin{proof}[Proof of Lemma~\ref{lem:rubin}]
  We may assume that $d(v) \geq 2$, as otherwise $G'$ is just
  a subgraph of $G$.
  Let $v'$ be the merged vertex in $G'$, and let $L'$ be a list
  assignment on $G'$ such that $\sizeof{L'(w)} = 4m$ for all $w \in
  V(G')$. Define a list assignment $L$ on $G$ as follows:
  \[ L(w) =
  \begin{cases}
    L'(v'),& \text{ if $w = v$ or $w \in N(v)$,} \\
    L'(w),& \text{ otherwise.}
  \end{cases} \] Since $G$ is $\ab{4m}{2m}$-choosable, it has some
  proper $L$-colouring $\phi$.  For all $w \in N(v)$, we have $\phi(w) \cap
  \phi(v) = \emptyset$.  Since $L(w) = L(v)$ and since $\phi(w),
  \phi(v) \subset L(v)$ with $\sizeof{\phi(w)} + \sizeof{\phi(v)} =
  \sizeof{L(v)}$, this implies that $\phi(w) = L(v) - \phi(v)$ for all
  $w \in N(v)$.  We define an $L'$-colouring $\phi'$ of $G'$ by putting
  $\phi'(v') = L(v)-\phi(v)$ and putting $\phi'(w) = \phi(w)$ for all
  $w \in V(G') - v'$. Since $\phi$ was a proper $L$-colouring, we see
  that $\phi'$ is a proper $L'$-colouring. As $L'$ was arbitrary, we
  conclude that $G'$ is $\ab{4m}{2m}$-choosable.
\end{proof}
\begin{corollary}
  If $\Theta_{2, 2r, 2s}$ is $\ab{4}{2}$-choosable, then $\Theta_{1, 2r-1, 2s-1}$
  is $\ab{4}{2}$-choosable.
\end{corollary}
\begin{proof}
  Applying the operation of Lemma~\ref{lem:rubin} to a vertex $v$ of degree~$3$
  transforms $\Theta_{2,2r,2s}$ into $\Theta_{1, 2r-1, 2s-1}$.
\end{proof}
It therefore suffices to show that $\Theta_{2,2r,2s}$ is
$\ab{4}{2}$-choosable for all $r,s \geq 1$. Similar techniques will
allow us to deal with cycles sharing a vertex or joined by a path.
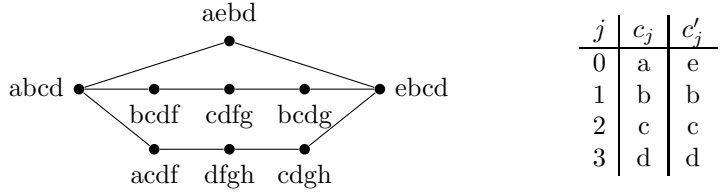
\begin{figure}
  \centering
  \begin{subfigure}{.6\textwidth}
  \begin{tikzpicture}[xscale=1, yscale=.8]
    \anonlist{left}{abcd} (u) at (0cm, 0cm) {};
    \anonlist{right}{ebcd} (v) at (4cm, 0cm) {};
    \anonlist{below}{acdf} (x1) at (1cm, -1cm) {};
    \anonlist{below}{dfgh} (x2) at (2cm, -1cm) {};
    \anonlist{below}{cdgh} (x3) at (3cm, -1cm) {};
    \anonlist{below}{bcdf} (y1) at (1cm, 0cm) {};
    \anonlist{below}{cdfg} (y2) at (2cm, 0cm) {};
    \anonlist{below}{bcdg} (y3) at (3cm, 0cm) {};
    \anonlist{above}{aebd} (z1) at (2cm, .8cm) {};
    \draw (u) -- (x1) -- (x2) -- (x3)  -- (v);
    \draw (u) -- (y1) -- (y2) -- (y3)  -- (v);
    \draw (u) -- (z1) -- (v);
  \end{tikzpicture}
  \end{subfigure}
  \begin{subfigure}{.2\textwidth}
    \begin{tabular}{r|c|c}
      $j$ & $c_j$ & $c'_j$ \\\hline
      0 & a & e \\
      1 & b & b \\
      2 & c & c \\
      3 & d & d
    \end{tabular}
  \end{subfigure}
  \caption{The graph $\Theta_{2,4,4}$, with a list assignment and an
    associated coupling. From top to bottom, the internal paths are
    $P^0$, $P^1$, and $P^2$.}
  \label{fig:theta244}
\end{figure}

We now 
%XD establish 
introduce some notation for various parts of theta graphs; Figure~\ref{fig:theta244} shows $\Theta_{2,4,4}$, as a reference.
\begin{definition}
  The vertices of degree $3$ in a theta graph are called $u$ and $v$.
  The \emph{internal paths} of a theta graph are the paths in
  $G-\{u,v\}$; the endpoints of the internal paths are the neighbors
  of $u$ and $v$.
\end{definition}
Fix a list assignment $L$, and let $L(u) = \{c_0, c_1, c_2, c_3\}$ and
$L(v) = \{c'_0, c'_1, c'_2, c'_3\}$, where the colours are indexed so
that $c'_j = c_j$ whenever $c_j \in L(u) \cap L(v)$. Note that this
indexing implies that $\{c_i, c'_i\} \cap \{c_j, c'_j\} = \emptyset$
whenever $i \neq j$.
\begin{definition}
  For a fixed indexing of $L(u)$ and $L(v)$, a \emph{couple} is a
  tuple of the form $(c_j, c'_j)$ for $j \in \{0,1,2,3\}$. When we
  write a couple, we suppress the parentheses and simply write
  $c_jc'_j$.  A \emph{pair} is a tuple $(S,T)$ with $S \subset L(u)$,
  $T \subset L(v)$, and $\sizeof{S} = \sizeof{T} = 2$.  A \emph{simple
    pair} is a pair $(S,T)$ such that for all $c_j \in S$, we also
  have $c'_j \in T$. A \emph{simple solution} is a simple pair $(S,T)$
  such that $\dam_{L,P}(S,T) \leq S_L(P) - 2\sizeof{V(P)}$ for all
  internal paths $P$.
\end{definition}
Observe that the definition of a couple and a simple pair depends on
the indexing of the colours of $L(u)$ and colours in $L(v)$. A simple
solution can be interpreted as a precolouring of $\{u,v\}$ which
extends (via Lemma~\ref{lem:abstract}) to all internal paths of the
theta graph. With any fixed indexing of $L(u)$ and $L(v)$, we first
try to find a simple solution.  We show that a simple solution exists
unless $L$ has a very specific form. Then we address this form as a
special case.

Equation~\eqref{eq:damformula} implies that if $S,T$ is a simple pair, then
\[ \dam_{L,P}(S,T) = \sum_{c_j \in S}\dam_{L,P}(\{c_j\}, \{c'_j\}). \]
In other words, when $(S,T)$ is a simple pair, we can simply calculate the damage of each
couple in $(S,T)$ independently, and add them together to obtain $\dam_{L,P}(S,T)$.
Moreover, for each $j$, we have $\dam_{L,P}(\{c_j\}, \{c'_j\})  \in \{0, 1, 2\}$.
\begin{definition}
  When $L$ is a list assignment on a theta graph,
  \begin{itemize}
  \item The couple $c_jc'_j$ is \emph{heavy} for the internal path $P$ if $\dam_{L,P}(\{c_j\}, \{c'_j\}) = 2$;
  \item The couple $c_jc'_j$ is \emph{light} for the internal path $P$ if $\dam_{L,P}(\{c_j\}, \{c'_j\}) = 1$;
  \item The couple $c_jc'_j$ is \emph{safe} for the internal path $P$ if $\dam_{L,P}(\{c_j\}, \{c'_j\}) = 0$.
  \end{itemize}
\end{definition}
\begin{definition}
  When $L$ is a list assignment on a theta graph, we say that an
  internal path $P$ \emph{blocks} a pair $(S,T)$ if $\dam_{L,P}(S,T) >
  S_L(P) - 2\sizeof{V(P)}$, i.e., if we cannot extend the partial
  colouring $\phi(u) = S$, $\phi(v) = T$ to all vertices of $P$.
\end{definition}
\begin{example}
  For the list assignment shown in Figure~\ref{fig:theta244}, the
  couple $\mathrm{ae}$ is heavy for $P^0$, safe for $P^1$, and light for $P^2$.
  The path $P^2$ blocks the simple pair $(\mathrm{ac}, \mathrm{ec})$.
\end{example}
% Observe that if $c_jc'_j$ is heavy for $P$, then $c_j \in \hat{X}_1$ and $c'_j
% \in \hat{X_n}$; if $c_jc'_j$ is light for $P$, then either
% $c_j = c'_j$ and $c_j \in A$, or else $\sizeof{\{c_j\} \cap \hat{X}_1}
% + \sizeof{\{c'_j\} \cap \hat{X}_n} = 1$.

Now we count how many simple pairs are blocked by each internal path.
It will be helpful to prove this lemma for more general theta graphs than $\Theta_{r,s,t}$.
\begin{lemma}\label{lem:intblock}
  Let $r_1, \ldots, r_k$ be positive integers, and let $L$ be a
  list assignment on $\Theta_{2r_1, \ldots, 2r_k}$.  Each internal path
  $P$ blocks at most $2$ simple pairs, and if $P$ blocks $2$ simple
  pairs, then $S_L(P) = 2\sizeof{V(P)} + 2$, and $P$ has one heavy
  couple and two light couples.
\end{lemma}
\begin{proof}
  Let $P$ be any internal path, and let $n = \sizeof{V(P)}$.  By the $m=1$ case of
  Lemma~\ref{lem:trivbound}, $S_L(P) \geq 2n + 2$.  If $S_L(P) \geq 2n
  + 4$, then $P$ does not block any simple pairs, since for any pair
  $(S,T)$, we have $\dam_{L,P}(S,T) \leq 4$. Hence it suffices to consider
  $S_L(P) \in \{2n+2, 2n+3\}$.

  We first argue that in both cases, $P$ has at most $2$ heavy couples. If
  $c_jc'_j$ is a heavy couple, then by Equation~\eqref{eq:damformula} we have
  \[\sizeof{\hat{X}_1 \cap \{c_j\}} + \sizeof{\hat{X}_n \cap \{c'_j\}} + \sizeof{A \cap \{c_j, c'_j\}} = 2. \] 
  In particular, since $\{c_i, c'_i\} \cap \{c_j, c'_j\} = \emptyset$ whenever $i \neq j$, we see that
  if $P$ has $3$ heavy couples, then $\sizeof{\hat{X}_1} + \sizeof{\hat{X}_n} + \sizeof{A} \geq 6$.
  By the $m=1$ case of Lemma~\ref{lem:hatsum}, this implies that $S_L(P) \geq 2n+4$.

  If $S_L(P) = 2n+3$, then $P$ blocks the simple pair $(S,T)$ only if
  $\dam_{L,P}(S,T) = 4$, i.e., if both couples used in $(S,T)$ are heavy.
  Since $P$ has at most $2$ heavy couples, this implies that $P$ blocks
  at most $1$ simple pair.

  If $S_L(P) = 2n+2$, then $P$ blocks the simple pair $(S,T)$ if and
  only if $\dam_{L,P}(S,T) \geq 3$, i.e., if one of the couples in
  $(S,T)$ is heavy and the other is not safe. Lemma~\ref{lem:hatsum}
  implies that if $P$ has $2$ heavy couples, then $P$ has no light
  couple, since that would imply that $\sizeof{\hat{X}_1} +
  \sizeof{\hat{X}_n} + \sizeof{A} \geq 5$. In particular, if $P$ has
  $2$ heavy couples, then it blocks at most $1$ simple pair. Likewise,
  if $P$ has $1$ heavy couple, then $P$ has at most $2$ light
  couples. The desired conclusion follows.
\end{proof}
Now we specialize to the $\Theta_{r,s,t}$ case.
\begin{corollary}\label{cor:exactblock}
  Let $r,s,t$ be positive integers, and let $L$ be a list
  assignment on $\Theta_{2r,2s,2t}$. If $L$ has no simple solution, then
  each simple pair $(S,T)$ is blocked by exactly one internal
  path. In particular, each couple $c_jc'_j$ is heavy for at most one
  internal path.
\end{corollary}
\begin{proof}
  There are $6$ simple pairs, and each of the three internal paths
  blocks at most $2$ of them; this proves the first part. If the
  couple $c_jc'_j$ is heavy for two different internal paths $P$ and
  $Q$, then since $P$ and $Q$ each have two light couples, there is
  some couple $c_kc'_k$ that is light for both $P$ and $Q$. Now the
  pair $(\{c_j,c_k\}, \{c'_j, c'_k\})$ is blocked by both $P$ and $Q$,
  contradicting the first part of the corollary.
\end{proof}
We now must handle the case where $L$ has no simple solution. First we
refine our notation.  By Corollary~\ref{cor:exactblock}, we may
reindex $L(u)$ and $L(v)$ so that for all $j \in \{0,1,2\}$, the
couple $c_jc'_j$ is heavy for $P^j$. (By simultaneously permuting the
labels in $L(u)$ and $L(v)$, this maintains the original property that
$c'_j = c_j$ whenever $c_j \in L(u) \cap L(v)$.) With this new
notation, we have the following further consequence of
Corollary~\ref{cor:exactblock}:
\begin{corollary}\label{cor:helix}
  Let $r,s,t$ be positive integers, and let $L$ be a list
  assignment on $\Theta_{2r,2s,2t}$. If $L$ has no simple solution, then
  $c_3c'_3$ is light for all internal paths $P^i$, and one of the two
  following situations must hold:
  \begin{enumerate}[(a)]
  \item $c_1c'_1$ is light for $P^0$, $c_2c'_2$ is light for $P^1$, and $c_0c'_0$ is light for $P^2$, or
  \item $c_2c'_2$ is light for $P^0$, $c_0c'_0$ is light for $P^1$, and $c_1c'_1$ is light for $P^2$.
  \end{enumerate}
\end{corollary}
\begin{proof}
  As in Corollary~\ref{cor:exactblock}, since there is no simple
  solution, each internal path blocks $2$ simple pairs. Thus, by
  Lemma~\ref{lem:intblock}, each internal path has one heavy couple
  and two light couples, and therefore has exactly one safe
  couple. For each $j \in \{0,1,2\}$, let $\pi(j)$ be the unique index
  in $\{0,1,2,3\}$ such that $c_{\pi(j)}c'_{\pi(j)}$ is safe for $P^j$. We
  will show that $\pi$ is a permutation of $\{0,1,2\}$ having no fixed
  points. It is clear that $\pi(j) \neq j$ for all $j \in \{0,1,2\}$,
  since $c_jc'_j$ is heavy for $P^j$.

  First we argue that $\pi$ is an injection. Suppose that $\pi(i) =
  \pi(j)$ for some $i \neq j$.  Since $c_jc'_j$ is heavy for $P^j$ and
  since $P^j$ has one heavy couple and two light couples, it follows
  that $c_ic'_i$ is light for $P^j$. Likewise, $c_jc'_j$ is light for
  $P^i$.  By Lemma~\ref{lem:intblock}, we have $S_L(P^i) -
  2\sizeof{V(P^i)} = S_L(P^j) - 2\sizeof{V(P^j)} = 2$, so the simple
  pair $(\{c_i,c_j\}, \{c'_i,c'_j\})$ is blocked by both $P^i$ and
  $P^j$, contradicting Corollary~\ref{cor:exactblock}.

  Next we argue that $\pi(j) \neq 3$ for all $j$. If $\pi(j) = 3$,
  then for both $i \in \{0,1,2\}-j$, the couple $c_ic'_i$ is light for
  $P^j$. Since $\pi$ is an injection, there is some $i \in
  \{0,1,2\}-j$ with $\pi(i) \neq j$, so that $c_jc'_j$ is light for
  $P^i$. Now the simple pair $(\{c_i,c_j\}, \{c'_i,c'_j\})$ is blocked
  by both $P^i$ and $P^j$, again contradicting
  Corollary~\ref{cor:exactblock}.

  Thus $\pi$ is a permutation of $\{0,1,2\}$ with no fixed points.
  This implies that $\pi$ is a $3$-cycle. If $\pi = (0\ 2\ 1)$
  then situation~(a) holds, and if $\pi = (0\ 1\ 2)$ then situation~(b)
  holds.
\end{proof}
\begin{corollary}
  If $r,s$ are positive integers, then $\Theta_{2,2r,2s}$ is $\ab{4}{2}$-choosable.
\end{corollary}
\begin{proof}
  Let $G = \Theta_{2, 2r, 2s}$, and let $L$ be any list assignment on
  $G$.  We must show that $G$ is $\ab{L}{2}$-colourable. If $L$ has a
  simple solution, then there is nothing more to show, so we may
  assume that $L$ does not have a simple solution. Let $P^0$, $P^1$,
  and $P^2$ be the internal paths of $G$, with $\sizeof{V(P^0)} =
  1$. We may choose the indexing of $P^1$ and $P^2$ so that
  situation~(a) of Corollary~\ref{cor:helix} holds (this is the case,
  for example, in Figure~\ref{fig:theta244}). For each $i \in
  \{0,1,2\}$, we write $\hat{X}_1^{i}$, $\hat{X}_n^{i}$, and $A^{i}$
  to refer to the sets $\hat{X}_1$, $\hat{X}_n$, and $A$ calculated
  for $P^i$.

  Since $\sizeof{V(P^0)} = 1$, we know that
  $\hat{X}_1^0 = \hat{X}_n^0 = \emptyset$. Hence, since $c_0c'_0$
  is heavy for $P_0$, we must have $c_0 \neq c'_0$.  Hence $c_0 \notin
  L(u) \cap L(v)$ and $c'_0 \notin L(u) \cap L(v)$.

  Now consider $P^2$. Since $c_0 \neq c'_0$ and $c_0c'_0$ is light for
  $P^2$, we must have either $c_0 \notin A^2 \cup \hat{X}_1^2$ or $c'_0 \notin
  A^2 \cup \hat{X}_n^2$. By symmetry, we may assume that $c_0 \notin A^2 \cup \hat{X}_1^2$. Let
  $S = \{c_0, c_3\}$ and let $T = \{c'_2, c'_3\}$.

  We check that $\dam_{L,P}(S,T) \leq 2$ for each internal path
  $P$. By Equation~\eqref{eq:damformula}, for each $i$ we have
  \begin{align*}
    \dam_{L,P^i}(S,T) &= \sizeof{(A^{i} \cup \hat{X}_1^{i}) \cap S} +
    \sizeof{(A^{i} \cup \hat{X}_n^{i}) \cap T} - \sizeof{A^{i} \cap S \cap T} \\
    &\leq \sizeof{(A^{i} \cup \hat{X}_1^{i}) \cap \{c_0\}} + \sizeof{(A^{i} \cup \hat{X}_n^{i}) \cap \{c'_2\}} + \\
    &\qquad \sizeof{(A^{i} \cup \hat{X}_1^{i}) \cap \{c_3\}} + \sizeof{(A^{i} \cup \hat{X}_n^{i}) \cap \{c'_3\}} -
    \sizeof{A^{i} \cap \{c_3\} \cap \{c'_3\}} \\
    &= \sizeof{(A^{i} \cup \hat{X}_1^{i}) \cap \{c_0\}} + \sizeof{(A^{i} \cup \hat{X}_n^{i}) \cap \{c'_2\}} + \dam_{L,P^i}(\{c_3\}, \{c'_3\}).
  \end{align*}
  Since the couple $c_3c'_3$ is light for all internal paths,
  we have $\dam_{L,P^i}(\{c_3\}, \{c'_3\}) = 1$ for all $P^i$, so that
  \[ \dam_{L,P^i}(S,T) \leq \sizeof{(A^{i} \cup \hat{X}_1^{i}) \cap
    \{c_0\}} + \sizeof{(A^{i} \cup \hat{X}_n^{i}) \cap \{c'_2\}} +
  1. \]
  Each term of this sum is clearly at most $1$, so to show that
  $\dam_{L, P^i}(S,T) \leq 2$, it suffices to show that one of the
  terms is $0$ for each $i$.

  Since $c_2c'_2$ is safe for $P^0$, we have $c'_2 \notin A^0 \cup
  \hat{X}^0_n$, so $\dam_{L, P^0}(S,T) \leq 2$. Likewise, since
  $c_0c'_0$ is safe for $P^1$, we have $c_0 \notin A^1 \cup
  \hat{X}_1^1$, so $\dam_{L, P^1}(S,T) \leq 2$. By assumption, $c_0
  \notin A^2 \cup \hat{X}_1^2$, so we also have $\dam_{L, P^2}(S,T)
  \leq 2$.
  %The case $c'_0 \notin A^2 \cup \hat{X}^2_n$ is similar; here, we take $S = \{c_2,
  %c_3\}$ and $T = \{c'_0, c'_3\}$.
\end{proof}

\section{Even Cycles Sharing a Vertex or Joined by a Path}\label{sec:evencycle}
In this section, we show that if $G$ consists of two cycles sharing a
single vertex or two vertex-disjoint cycles joined by a path, then $G$
is $\ab{4}{2}$-choosable.  In fact, one can show that these graphs are
$\ab{4m}{2m}$-choosable for all $m$; in the interest of brevity, we prove
only the $\ab{4}{2}$-choosability case, which allows us to reuse some
tools from the previous section. As before, whenever $L$ is a list assignment,
we tacitly assume $\sizeof{L(v)} = 4$ for all $v \in V(G)$.
\begin{definition}
  Let $P$ be a path with an odd number of vertices, let $L$ be a list
  assignment on $P$, and let $W$ be a set of $4$ colours. An
  \emph{$L$-bad $W$-set} for $P$ is a set $S \subset W$ of $2$ colours
  such that $\dam_{L,P}(S,S) > S_L(P) - 2\sizeof{V(P)}$. When $L$ is
  understood, we abbreviate ``$L$-bad $W$-set'' to ``bad $W$-set''.
\end{definition}
\begin{lemma}\label{lem:twobad}
  If $P$ is a path with an odd number of vertices, $L$ is a list
  assignment on $P$, and $W$ is any set of $4$ colours, then $P$ has
  at most $2$ $L$-bad $W$-sets.
\end{lemma}
\begin{proof}
  Consider the graph $H$ obtained by adding new vertices $u$ and $v$
  on the ends of $P$, and extend $L$ to $V(H)$ by putting $L(u) = L(v)
  = W$. Considering $H$ as a theta graph with $P$ as its only internal
  path (as in Section~\ref{sec:goodtheta}), we see that $S$ is a bad
  set for $P$ if and only if $P$ blocks the simple pair $(S,S)$. By
  Lemma~\ref{lem:intblock}, it follows that $P$ has at most $2$ bad
  sets.
\end{proof}
\begin{lemma}\label{lem:inject}
  Let $Q$ be a path with endpoints $u$ and $v$. For every list
  assignment $L$ on $Q$, there is an injective function $h : {L(u)
    \choose 2} \to {L(v) \choose 2}$ such that for all $S \in {L(u)
    \choose 2}$, the precolouring $\phi(u) = S$, $\phi(v) = h(S)$
  extends to all of $Q$.
\end{lemma}
\begin{proof}
  We use induction on $\sizeof{V(Q)}$. When $\sizeof{V(Q)} = 1$
  or $\sizeof{V(Q)} = 2$, the claim clearly holds: when
  $\sizeof{V(Q)} = 1$ we may take $h$ to be the identity
  function, and when $\sizeof{V(Q)} = 2$ it suffices that
  $S \cap h(S) = \emptyset$ for all $S$; such an $h$
  is easy to construct.

  Hence we may assume that $\sizeof{V(Q)} > 2$ and the
  claim holds for smaller paths. Let $v'$ be the unique
  neighbor of $v$. We split $Q$ into the $u,v'$-subpath
  $Q_1$ and the $v',v$-subpath $Q_2$, overlapping only at $v'$.
  Let $h_1$ and $h_2$ be the functions for $Q_1$ and $Q_2$
  respectively, as guaranteed by the induction hypothesis.
  Composing $h_2$ and $h_1$, we see that $h_2 \after h_1$
  has the desired properties.
\end{proof}
We handle ``two cycles sharing a vertex'' as a special case of
``two cycles joined by a path'', considering the shared vertex
as a path on $1$ vertex.
\begin{corollary}
  If $G$ is a graph consisting of two even cycles joined
  by a (possibly-trivial) path, then $G$ is $\ab{4}{2}$-choosable.
\end{corollary}
\begin{proof}
  Let $C$ and $D$ be the cycles in $G$, and let $u \in V(C)$ and $v
  \in V(D)$ be the endpoints of the path joining $C$ and $D$.  Let $P
  = C-u$, let $R = D - v$, and let $Q$ be the path joining $u$ and
  $v$, so that $P, Q, R$ are disjoint paths with $V(P) \cup V(Q) \cup
  V(R) = V(G)$. The situation is illustrated in Figure~\ref{fig:cyclepath}.
  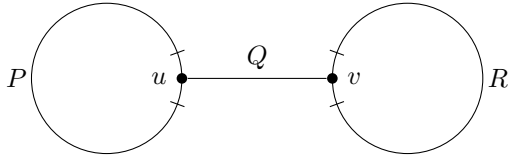
\begin{figure}
    \centering
    \begin{tikzpicture}
      \begin{scope}
        \draw (0cm,0cm) circle (1cm);
        \draw (20:.9cm) -- (20:1.1cm);
        \draw (-20:.9cm) -- (-20:1.1cm);
      \end{scope}
      \begin{scope}[xshift=4cm, xscale=-1]
        \draw (0cm,0cm) circle (1cm);
        \draw (20:.9cm) -- (20:1.1cm);
        \draw (-20:.9cm) -- (-20:1.1cm);
      \end{scope}
      \lpoint{u} (u) at (1cm,0cm) {};
      \rpoint{v} (v) at (3cm,0cm) {};
      \draw (u) -- (v) node[above, pos=0.5] {$Q$};
      \node at (-1.2cm, 0cm) {$P$};
      \node at (5.2cm, 0cm) {$R$};
    \end{tikzpicture}
    \caption{Decomposing $G$ into $P,Q,R$.}
    \label{fig:cyclepath}
  \end{figure}
  By Lemma~\ref{lem:twobad}, the path $P$ has at most two bad
  $L(u)$-sets, and the path $R$ has at most two bad $L(v)$-sets.
  Let $h : {L(u) \choose 2} \to {L(v) \choose 2}$ be the injection
  guaranteed by Lemma~\ref{lem:inject}. Since there are $6$ ways to
  choose a set $S \in {L(u) \choose 2}$, we see that there is some $S$
  such that $S$ is not bad for $P$ and $h(S)$ is not bad for $Q$. It
  follows that we may extend the precolouring $\phi(u) = S$, $\phi(v) = h(S)$
  to all of $P$, $Q$, and $R$.
\end{proof}

Tuza and Voigt have already shown~\cite{tuzavoigt-k24} that $K_{2,4}$
is $\ab{4m}{2m}$-choosable for all $m$, so this completes the positive
direction    of Theorem~\ref{thm:main}.
\section{Non-$\ab{4}{2}$-Choosable Theta Graphs}\label{sec:ctx}
In this section, we argue that if $\min\{r,s,t\} \geq 3$, then
$\Theta_{r,s,t}$ is not $\ab{4}{2}$-choosable, and that if $t \geq 2$,
then $\Theta_{2,2,2,2t}$ is not
$\ab{4}{2}$-choosable. Figure~\ref{fig:ctx} shows noncolourable list
assignments for $\Theta_{2,2,2,4}$ and $\Theta_{3,3,3}$.
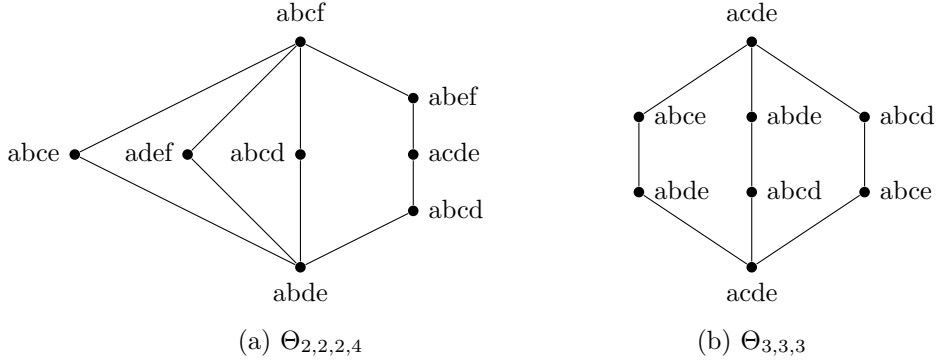
\begin{figure}
  \centering
  \begin{tikzpicture}
    \node at (-3cm, -4cm) {(a) $\Theta_{2,2,2,4}$};
    \node at (3cm, -4cm) {(b) $\Theta_{3,3,3}$};
    \begin{scope}[xshift=-3cm, yscale=1.5, xscale=1.5]
      \anonlist{above}{abcf} (u) at (0cm, 0cm) {};
      \anonlist{left}{abce} (w1) at (-2cm, -1cm) {};
      \anonlist{left}{adef} (w2) at (-1cm, -1cm) {};
      \anonlist{left}{abcd} (w3) at (0cm, -1cm) {};
      \anonlist{right}{abef} (v1) at (1cm, -.5cm) {};
      \anonlist{right}{acde} (v2) at (1cm, -1cm) {};
      \anonlist{right}{abcd} (v3) at (1cm, -1.5cm) {};
      \anonlist{below}{abde} (z) at (0cm, -2cm) {};
      \draw (u) -- (w1) -- (z);
      \draw (u) -- (w2) -- (z);
      \draw (u) -- (w3) -- (z);
      \draw (u) -- (v1) -- (v2) -- (v3) -- (z);
    \end{scope}
    \begin{scope}[xshift=3cm, yscale=1, xscale=1.5]
      \anonlist{above}{acde} (u) at (0cm, 0cm) {};
      \anonlist{below}{acde} (v) at (0cm, -3cm) {};
      \anonlist{right}{abce} (x1) at (-1cm, -1cm) {};
      \anonlist{right}{abde} (x2) at (-1cm, -2cm) {};
      \anonlist{right}{abde} (y1) at (0cm, -1cm) {};
      \anonlist{right}{abcd} (y2) at (0cm, -2cm) {};
      \anonlist{right}{abcd} (z1) at (1cm, -1cm) {};
      \anonlist{right}{abce} (z2) at (1cm, -2cm) {};
      \draw (u) -- (x1) -- (x2) -- (v);
      \draw (u) -- (y1) -- (y2) -- (v);
      \draw (u) -- (z1) -- (z2) -- (v);
    \end{scope}
  \end{tikzpicture}\\
  \caption{Noncolourable list assignments for $\Theta_{2,2,2,4}$ and $\Theta_{3,3,3}$.}
  \label{fig:ctx}
\end{figure}

To show that larger theta graphs are not $\ab{4}{2}$-choosable, we
again apply Lemma~\ref{lem:rubin}. In particular, the contrapositive
of Lemma~\ref{lem:rubin} states that if $G'$ is not
$\ab{4}{2}$-choosable, then $G$ is not $\ab{4}{2}$ choosable either.
Hence $\Theta_{4,4,4}$ is not $\ab{4}{2}$-choosable, since
$\Theta_{3,3,3}$ is obtained from $\Theta_{4,4,4}$ by applying this
reduction to a vertex of degree $3$.

Likewise, $\Theta_{2,2,2,2t}$ is obtained from $\Theta_{2,2,2,2t+2}$
by applying this reduction to an interior vertex of the   path of length $2t+1$; hence, since
$\Theta_{2,2,2,4}$ is not $\ab{4}{2}$-choosable, it follows by
induction on $t$ that when $t \geq 2$, the graph $\Theta_{2,2,2,2t}$
is not $\ab{4}{2}$-choosable.  Similarly, since $\Theta_{3,3,3}$ is
not $\ab{4}{2}$-choosable, no graph of the form $\Theta_{2r+1, 2s+1,
  2t+1}$ for $r,s,t \geq 1$ is $\ab{4}{2}$-choosable, and since
$\Theta_{4,4,4}$ is not $\ab{4}{2}$-choosable, no graph of the form
$\Theta_{2r,2s,2t}$ for $r,s,t \geq 2$ is $\ab{4}{2}$-choosable.

\section{A Conjecture of Voigt}\label{sec:voigt}
Voigt~\cite{voigt} conjectured that every bipartite
$3$-choosable-critical graph is $\ab{4m}{2m}$-choosable for all
$m$. We have seen that this conjecture fails for $m=1$: there are
bipartite $3$-choosable-critical graphs which are not
$\ab{4}{2}$-choosable. In this section, we prove the following weaker
version of Voigt's conjecture:
\begin{theorem}\label{thm:weakvoigt}
  There is a fixed integer $k$ such that for every positive integer $m$,
  every bipartite $3$-choosable-critical graph is
  $\ab{4km}{2km}$-choosable.
\end{theorem}
Our proof is based on the following theorem of Alon, Tuza, and Voigt~\cite{atv}.
\begin{theorem}[Alon--Tuza--Voigt~\cite{atv}]\label{thm:atv}
  For every integer $n$ there exists a number $f(n) \leq (n+1)^{2n+2}$
  such that the following holds. For every graph $G$ with $n$ vertices
  and with fractional chromatic number $\chi^*$, and for every integer
  $M$ which is divisible by all integers from $1$ to $f(n)$, $G$ is
  $\ab{M}{M/\chi^*}$-choosable.
\end{theorem}
Lemma~\ref{lem:hatsum} and Lemma~\ref{lem:trivbound} suggest that when
$n$ is odd, the ``worst case'' tuples $(A, \hat{X}_1, \hat{X}_n)$ are
those satisfying $\sizeof{A} + \sizeof{\hat{X}_1} + \sizeof{\hat{X}_n}
= 4m$.  The following lemma shows that any such sets can be
``realized'' on a path of length $3$:
\begin{lemma}\label{lem:realize}
  Let $P$ be a path on $3$ vertices, and
  let $B$, $Y$, $Z$ be sets such that $B \cap Y = \emptyset$, $B \cap
  Z = \emptyset$, and $\sizeof{B} + \sizeof{Y} + \sizeof{Z} =
  4m$. There exists a list assignment $L$ on $P$ such that:
  \begin{itemize}
  \item $\sizeof{L(v)} = 4m$ for all $v \in V(P_3)$, and
  \item $(A, \hat{X}_1, \hat{X}_3) = (B, Y, Z)$, and
  \item $S_L(P) = 8m$.
  \end{itemize}
\end{lemma}
\begin{proof}
  Let $J_1$ and $J_2$ be sets disjoint from each other and disjoint
  from $B \cup Y \cup Z$ such that
  \begin{align*}
    \sizeof{J_1} &= 4m - \sizeof{B} - \sizeof{Y},\\
    \sizeof{J_2} &= 4m - \sizeof{B} - \sizeof{Z}.
  \end{align*}
  Observe that
  \[ \sizeof{B} + \sizeof{J_1} + \sizeof{J_2} = 8m - \sizeof{B} - \sizeof{Y} - \sizeof{Z} = 4m. \]
  Let $v_1, v_2, v_3$ be the vertices of $P$ written in order, and consider
  the following list assignment:
  \begin{align*}
    L(v_1) &= B \cup Y \cup J_1, \\
    L(v_2) &= B \cup J_1 \cup J_2, \\
    L(v_3) &= B \cup Z \cup J_2.
  \end{align*}
  It is easy to verify that $L$ has the desired properties.
\end{proof}
Lemma~\ref{lem:realize} allows us to obtain a partial converse
of Lemma~\ref{lem:rubin}, subject to certain restrictions on the choice
of the vertex $v$.
\begin{lemma}\label{lem:p5}
  Let $G$ be a graph containing a path $P$ on $5$ vertices which all
  have degree $2$ in $G$, and let $G'$ be the graph obtained by
  deleting the middle vertex of $P$ and merging its neighbors. The
  original graph $G$ is $\ab{4m}{2m}$-choosable if and only if the
  merged graph $G'$ is $\ab{4m}{2m}$-choosable.
\end{lemma}
\begin{proof}
  By Lemma~\ref{lem:rubin}, it suffices to show that if $G'$
  is $\ab{4m}{2m}$-choosable, then $G$ is $\ab{4m}{2m}$-choosable.
  Let $v_1, \ldots, v_5$ be the vertices of $P$, written in order.
  Let $P'$ be the $3$-vertex path in $G'$ corresponding to $P$,
  and let $v'_1, v'_2, v'_3$ be the vertices of $G'$, so that
  $v'_1 = v_1$ and $v'_3 = v_5$.

  Let $L$ be any list assignment for $G$ such that $\sizeof{L(v)} =
  4m$ for all $v \in V(G)$, and let $A,\hat{X}_1, \hat{X}_5$ be
  computed relative to $P$. We will define sets $B,Y,Z$
  based on $A, \hat{X}_1, \hat{X}_5$ and apply Lemma~\ref{lem:realize}
  to obtain a list assignment $L'$ on the shorter path $P'$.
  The definition is slightly different depending on whether
  $\sizeof{A} + \sizeof{\hat{X}_1} + \sizeof{\hat{X}_5} \leq 4m$:
  we either arbitrarily add elements or arbitrarily remove elements
  in order to reach the desired sum.
  \begin{itemize}
  \item When $\sizeof{A} + \sizeof{\hat{X}_1} + \sizeof{\hat{X}_5}
    \leq 4m$, let $B,Y,Z$ be arbitrary supersets of $A, \hat{X}_1,
    \hat{X}_n$ respectively such that $B \cap Y = \emptyset$, $B \cap
    Z = \emptyset$, and $\sizeof{B} + \sizeof{Y} + \sizeof{Z} = 4m$.
  \item When $\sizeof{A} + \sizeof{\hat{X}_1} + \sizeof{\hat{X}_5} > 4m$,
    let $B,Y,Z$ be arbitrary subsets of $A, \hat{X}_1, \hat{X}_5$ respectively,
    such that $\sizeof{B} + \sizeof{Y} + \sizeof{Z} = 4m$.
  \end{itemize}
  In either case, we may apply Lemma~\ref{lem:realize} to obtain a
  list assignment $L'$ on the shorter path $P'$ such that:
  \begin{itemize}
  \item $\sizeof{L'(v)} = 4m$ for all $v \in V(P')$, and
  \item $(A', \hat{X}'_1, \hat{X}'_3) = (B,Y,Z)$, and
  \item $S_{L'}(P') = 8m$.
  \end{itemize}
  We extend $L'$ to all of $G'$ by defining $L'(v) = L(v)$ for $v
  \notin V(P')$.

  Let $G_0 = G' - V(P') = G - V(P)$, and let $w,z$ be the neighbors of
  $v'_1, v'_3$ in $G_0$, respectively.  Since $G'$ is
  $\ab{4m}{2m}$-choosable, Lemma~\ref{lem:abstract} says there is a
  proper $(L':2m)$-colouring $\phi$ of $G_0$ such that $\dam_{L',P'}(\phi(w),
  \phi(z)) \leq 2m$.

  If $\sizeof{A} + \sizeof{\hat{X}_1} + \sizeof{\hat{X}_n} \leq 4m$,
  then Equation~\eqref{eq:damformula} yields
  \[ \dam_{L,P}(\phi(w), \phi(z)) \leq \dam_{L', P'}(\phi(w), \phi(z)) \leq 2m, \]
  while if $\sizeof{A} + \sizeof{\hat{X}_1} + \sizeof{\hat{X}_n} = 4m+c$
  for some $c > 0$, then Equation~\eqref{eq:damformula} yields
  \[ \dam_{L,P}(\phi(w), \phi(z)) \leq \dam_{L',P'}(\phi(w), \phi(z))+c \leq 2m+c. \]
  Applying Lemma~\ref{lem:trivbound} in the first case and Lemma~\ref{lem:hatsum}
  in the second, we obtain
 \[ \dam_{L,P}(\phi(w), \phi(z))
  \leq S_L(P) - 10m. \] Applying Lemma~\ref{lem:abstract} in the other
  direction, we see that $G$ is $\ab{L}{2m}$-colourable.  Since $L$
  was arbitrary, $G$ is $\ab{4m}{2m}$-choosable.
\end{proof}
\begin{proof}[Proof of Theorem~\ref{thm:weakvoigt}]
  There are only finitely many bipartite $3$-choosable-critical graphs
  which are minimal with respect to the reduction of
  Lemma~\ref{lem:p5}.  In particular, all such graphs have at most
  $14$ vertices, the largest such graph being
  $\Theta_{5,5,5}$. Let $f$ be the function given by Theorem~\ref{thm:atv},
  and let $\fmax = \max\{f(n) \st n \leq 14\}$.

  By Theorem~\ref{thm:atv}, if $k/4$ is divisible by all numbers up to $\fmax$,
  then all minimal bipartite $3$-choosable-critical graphs are
  $\ab{4k}{2k}$ choosable.  In particular, fixing the smallest such
  $k$ and applying Lemma~\ref{lem:p5}, we see that all bipartite
  $3$-choosable-critical graphs are $\ab{4km}{2km}$-choosable for all
  $m$.
\end{proof}
\section{Characterizing the $\ab{4}{2}$-Choosable Graphs: A Conjecture}
\label{sec:all42}
Having determined which 3-choosable-critical graphs are
$\ab{4}{2}$-choosable, the next natural step in investigating
$\ab{4}{2}$-choosability is to characterize \emph{all}
$\ab{4}{2}$-choosable graphs, mirroring Rubin's characterization of
the $2$-choosable graphs~\cite{rubin}. As Theorem~\ref{thm:main}
shows, the $\ab{4}{2}$-choosable graphs have considerably more variety
than the $2$-choosable graphs, so the proof of any such
characterization is likely to be much more involved than Rubin's
proof.

 Rubin observed that $G$ is
$2$-choosable if and only if its core is $2$-choosable, and the same
observation holds for $\ab{4}{2}$-choosability. It clearly
also suffices to consider only connected graphs, so we restrict to
the case where $G$ is connected with minimum degree at least $2$.
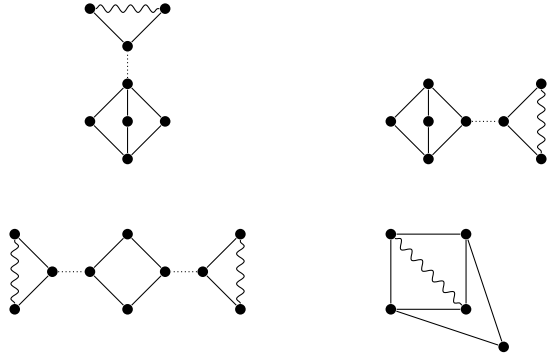
\begin{figure}
  \centering
  \begin{tikzpicture}
    \begin{scope}[xshift=-2cm, yshift=2cm]
      \apoint{} (u) at (0cm, 0cm) {};
      \apoint{} (w) at (0cm, .5cm) {};
      \apoint{} (v) at (0cm, -1cm) {};
      \apoint{} (x1) at (0cm, -.5cm) {};
      \apoint{} (x2) at (.5cm, -.5cm) {};
      \apoint{} (x3) at (-.5cm, -.5cm) {};
      \apoint{} (y1) at (-.5cm, 1cm) {};
      \apoint{} (y2) at (.5cm, 1cm) {};
      \draw (u) -- (x1) -- (v) -- (x2) -- (u) -- (x3) -- (v);
      \draw (y1) -- (w) -- (y2);
      \draw[arbpath] (u) -- (w);
      \draw[stretchy] (y1) -- (y2);
    \end{scope}
    \begin{scope}[xshift=2cm, yshift=2cm]
      \apoint{} (u) at (0cm, 0cm) {};
      \apoint{} (v) at (0cm, -1cm) {};
      \apoint{} (x1) at (0cm, -.5cm) {};
      \apoint{} (x2) at (.5cm, -.5cm) {};
      \apoint{} (t) at (1cm, -.5cm) {};
      \apoint{} (x3) at (-.5cm, -.5cm) {};
      \apoint{} (y1) at (1.5cm, 0cm) {};
      \apoint{} (y2) at (1.5cm, -1cm) {};
      \draw (u) -- (x1) -- (v) -- (x2) -- (u) -- (x3) -- (v);
      \draw (y1) -- (t) -- (y2);
      \draw[arbpath] (x2) -- (t);
      \draw[stretchy] (y1) -- (y2);
    \end{scope}
    \begin{scope}[xshift=-2cm, yshift=-.5cm]
      \begin{scope}[xshift=-1.5cm]
        \apoint{} (v) at (1cm, 0cm) {};
        \apoint{} (vp) at (.5cm, 0cm) {};
        \apoint{} (w) at (2cm, 0cm) {};
        \apoint{} (wp) at (2.5cm, 0cm) {};

        \apoint{} (x1) at (0cm, -.5cm) {};
        \apoint{} (x2) at (0cm, .5cm) {};
        \apoint{} (y1) at (1.5cm, -.5cm) {};
        \apoint{} (y2) at (1.5cm, .5cm) {};
        \apoint{} (t1) at (3cm, -.5cm) {};
        \apoint{} (t2) at (3cm, .5cm) {};
        \draw[stretchy] (x1) -- (x2);
        \draw[stretchy] (t1) -- (t2);
        \draw (w) -- (y1) -- (v) -- (y2) -- (w);
        \draw (x1) -- (vp) -- (x2);
        \draw (t1) -- (wp) -- (t2);
        \draw[arbpath] (vp) -- (v);
        \draw[arbpath] (wp) -- (w);
      \end{scope}
    \end{scope}
    \begin{scope}[xshift=2cm, yshift=-.5cm]
      \begin{scope}[yscale=-1]
        \apoint{} (c1) at (-.5cm, -.5cm) {};
        \apoint{} (c2) at (.5cm, -.5cm) {};
        \apoint{} (c3) at (.5cm, .5cm) {};
        \apoint{} (c4) at (-.5cm, .5cm) {};
        \draw (c1) -- (c2) -- (c3) -- (c4) -- (c1);
        \draw[stretchy] (c1) -- (c3);
        \apoint{} (d) at (1cm, 1cm) {};
        \draw (c2) -- (d) -- (c4);
      \end{scope}
    \end{scope}
  \end{tikzpicture}
  \caption{Exceptional graphs in Conjecture~\ref{coj:ourconj}. Wavy
    lines represent paths with an arbitrary odd number of vertices.
    Dotted lines represent paths with an arbitrary number (any
    parity) of vertices, possibly $1$.}
  \label{fig:manygraphs}
\end{figure}
\begin{figure}
  \centering
  \begin{tikzpicture}[xshift=-2cm, yshift=-.5cm]
    \begin{scope}[xshift=-1.5cm]
      \apoint{} (v) at (1cm, 0cm) {};
      \apoint{} (w) at (2cm, 0cm) {};
      \apoint{} (wp) at (2.5cm, 0cm) {};

      \apoint{} (x1) at (.5cm, -.5cm) {};
      \apoint{} (x2) at (.5cm, .5cm) {};
      \apoint{} (x0) at (0cm, 0cm) {};

      \apoint{} (y1) at (1.5cm, -.5cm) {};
      \apoint{} (y2) at (1.5cm, .5cm) {};
      \apoint{} (t1) at (3cm, -.5cm) {};
      \apoint{} (t2) at (3cm, .5cm) {};
      \apoint{} (s1) at (3.5cm, -.5cm) {};
      \apoint{} (s2) at (4cm, 0cm) {};
      \apoint{} (s3) at (3.5cm, .5cm) {};

      \draw (x1) -- (x0) -- (x2) -- (v) -- (x1);
      \draw (w) -- (y1) -- (v) -- (y2) -- (w);
      \draw(wp) -- (t1) -- (s1) -- (s2) -- (s3) -- (t2) -- (wp);
      \draw (wp) -- (w);
    \end{scope}
  \end{tikzpicture}
  \caption{One possible realization of the lower-left graph in Figure~\ref{fig:manygraphs}.}
  \label{fig:example}
\end{figure}
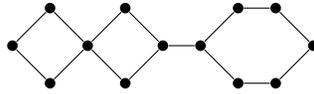
\begin{conjecture}\label{coj:ourconj}
  If $G$ is a connected graph with $\delta(G) \geq 2$, then $G$ is
  $\ab{4}{2}$-choosable if and only if one of the following holds:
  \begin{itemize}
  \item $G$ is $2$-choosable, or
  \item $G$ is one of the $3$-choosable-critical graphs listed in
    Theorem~\ref{thm:main}, or
  \item $G$ is one of the exceptional graphs shown in
    Figure~\ref{fig:manygraphs}.
  \end{itemize}
\end{conjecture}
Figure~\ref{fig:manygraphs} contains some complex visual notation used
to represent parameterized families of graphs; Figure~\ref{fig:example}
shows an example of how to interpret this notation.

Conjecture~\ref{coj:ourconj} is supported by substantial evidence.
Through computer search, we determined that among all graphs with at
most $9$ vertices, only the graphs given by
Conjecture~\ref{coj:ourconj} are $\ab{4}{2}$-choosable. It appears
that all graphs with a larger number of vertices are either one of the
$\ab{4}{2}$-choosable graphs listed in Conjecture~\ref{coj:ourconj},
or contain some subgraph already known to be
non-$\ab{4}{2}$-choosable.

A list of ``small'' minimal non-$\ab{4}{2}$-choosable graphs, each
with a nonchoosable list assignment, is given in
Figure~\ref{fig:badgraphs}. Each of the list assignments was found by
computer search.  The variety of these graphs represents a significant
obstruction to any proof of Conjecture~\ref{coj:ourconj}, which would
seem to require a correspondingly complex structure theorem. While we
believe that such a proof could be found, it would likely be quite
long and beyond the scope of this paper.

The computer analysis for the positive direction of
Conjecture~\ref{coj:ourconj} is based on
Lemma~\ref{lem:abstract}. Each of the graphs in
Figure~\ref{fig:manygraphs} has a small set of vertices $X$ such that
$G-X$ is a linear forest, with only the endpoints of its paths having
neighbors in $X$.  Rather than generating all list assignments
for the entire graph $G$, it suffices to generate all list assignments
for $X$, and for each list assignment, to generate the possible tuples
$(A, \hat{X}_1, \hat{X}_n)$ for each of the paths in $G-X$. For each
such tuple, we then search for a partial colouring $\phi$ of $G[X]$
that satisfies the hypothesis of Lemma~\ref{lem:abstract}.

However, we have not been able to find a human-readable proof that the
exceptional graphs in Conjecture~\ref{coj:ourconj} are indeed
$\ab{4}{2}$-choosable, nor have we been able to prove the structure
theorem alluded to above.
% (and without a formal proof of the result, we
% retain a healthy scepticism regarding the correctness of the computer program).
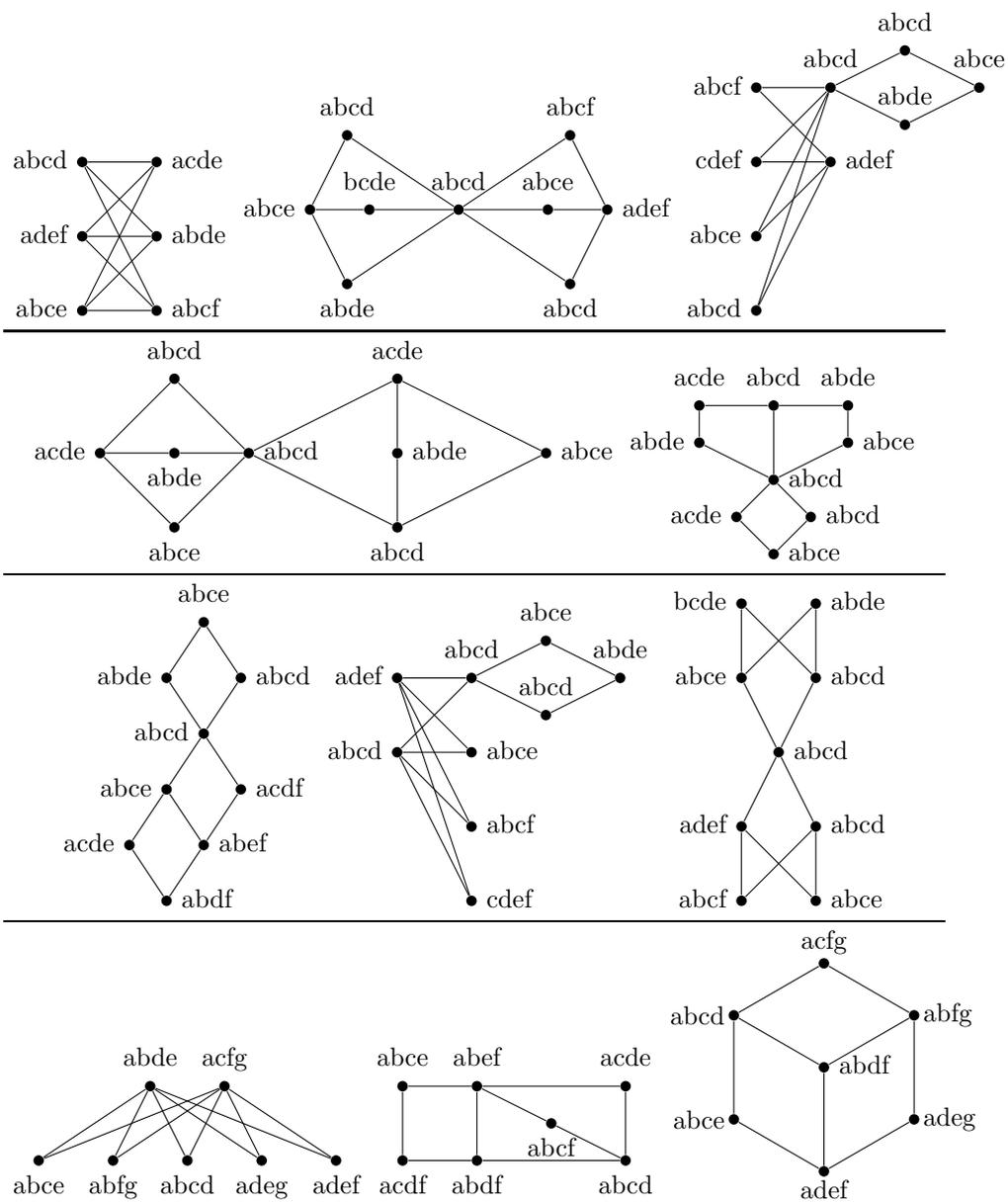
\begin{figure}
  \centering
  \begin{tikzpicture}
    \anonlist{left}{adef} (v0) at (0cm, 1cm) {};
    \anonlist{right}{abcf} (v1) at (1cm, 0cm) {};
    \anonlist{right}{abde} (v2) at (1cm, 1cm) {};
    \anonlist{right}{acde} (v3) at (1cm, 2cm) {};
    \anonlist{left}{abcd} (v4) at (0cm, 2cm) {};
    \anonlist{left}{abce} (v5) at (0cm, 0cm) {};
    \draw (v0) -- (v1);
    \draw (v0) -- (v2);
    \draw (v0) -- (v3);
    \draw (v1) -- (v4);
    \draw (v1) -- (v5);
    \draw (v2) -- (v4);
    \draw (v2) -- (v5);
    \draw (v3) -- (v4);
    \draw (v3) -- (v5);
  \end{tikzpicture}%\\
  \begin{tikzpicture}[scale=1]
    \anonlist{above}{abcd} (v0) at (-1.5cm, 1cm) {};
    \anonlist{above}{abce} (v1) at (1.2cm, 0cm) {};
    \anonlist{above}{abcf} (v2) at (1.5cm, 1cm) {};
    \anonlist{below}{abcd} (v3) at (1.5cm, -1cm) {};
    \anonlist{below}{abde} (v4) at (-1.5cm, -1cm) {};
    \anonlist{above}{bcde} (v5) at (-1.2cm, 0cm) {};
    \anonlist{left}{abce} (v6) at (-2cm, 0cm) {};
    \anonlist{right}{adef} (v7) at (2cm, 0cm) {};
    \anonlist{above}{abcd} (v8) at (0cm, 0cm) {};
    \draw (v0) -- (v8);
    \draw (v0) -- (v6);
    \draw (v1) -- (v8);
    \draw (v1) -- (v7);
    \draw (v2) -- (v8);
    \draw (v2) -- (v7);
    \draw (v3) -- (v8);
    \draw (v3) -- (v7);
    \draw (v4) -- (v8);
    \draw (v4) -- (v6);
    \draw (v5) -- (v8);
    \draw (v5) -- (v6);
  \end{tikzpicture}%\\
  \begin{tikzpicture}
    \anonlist{above}{abcd} (v0) at (1cm, .5cm) {};
    \anonlist{above}{abde} (v1) at (1cm,  -.5cm) {};
    \anonlist{left}{abcf} (v2) at (-1cm, 0.0cm) {};
    \anonlist{left}{cdef} (v3) at (-1cm, -1cm) {};
    \anonlist{left}{abce} (v4) at (-1cm, -2cm) {};
    \anonlist{left}{abcd} (v5) at (-1cm, -3cm) {};
    \anonlist{above}{abce} (v6) at (2.0cm, 0cm) {};
    \anonlist{right}{adef} (v7) at (0cm, -1cm) {};
    \anonlist{above}{abcd} (v8) at (0cm, 0cm) {};
    \draw (v0) -- (v8);
    \draw (v0) -- (v6);
    \draw (v1) -- (v8);
    \draw (v1) -- (v6);
    \draw (v2) -- (v8);
    \draw (v2) -- (v7);
    \draw (v3) -- (v8);
    \draw (v3) -- (v7);
    \draw (v4) -- (v8);
    \draw (v4) -- (v7);
    \draw (v5) -- (v8);
    \draw (v5) -- (v7);
  \end{tikzpicture}\\\hrule
  \begin{tikzpicture}
    \anonlist{above}{abcd} (v0) at (-1cm, 1cm) {};
    \anonlist{below}{abde} (v1) at (-1cm, 0cm) {};
    \anonlist{below}{abce} (v2) at (-1cm, -1cm) {};
    \anonlist{right}{abce} (v3) at (4cm, 0cm) {};
    \anonlist{right}{abde} (v4) at (2cm, 0cm) {};
    \anonlist{above}{acde} (v5) at (2cm, 1cm) {};
    \anonlist{below}{abcd} (v6) at (2cm, -1cm) {};
    \anonlist{left}{acde} (v7) at (-2cm, 0cm) {};
    \anonlist{right}{abcd} (v8) at (0cm, 0cm) {};
    \draw (v0) -- (v8);
    \draw (v0) -- (v7);
    \draw (v1) -- (v8);
    \draw (v1) -- (v7);
    \draw (v2) -- (v8);
    \draw (v2) -- (v7);
    \draw (v3) -- (v5);
    \draw (v3) -- (v6);
    \draw (v4) -- (v5);
    \draw (v4) -- (v6);
    \draw (v5) -- (v8);
    \draw (v6) -- (v8);
  \end{tikzpicture}%\\
  \begin{tikzpicture}
    \anonlist{right}{abcd} (v0) at (.5cm, -.5cm) {};
    \anonlist{left}{acde} (v1) at (-.5cm, -.5cm) {};
    \anonlist{right}{abce} (v2) at (1cm, .5cm) {};
    \anonlist{left}{abde} (v3) at (-1cm, .5cm) {};
    \anonlist{above}{abde} (v4) at (1cm, 1cm) {};
    \anonlist{above}{acde} (v5) at (-1cm, 1cm) {};
    \anonlist{right}{abce} (v6) at (0cm, -1cm) {};
    \anonlist{above}{abcd} (v7) at (0cm, 1cm) {};
    \anonlist{right}{abcd} (v8) at (0cm, 0cm) {};
    \draw (v0) -- (v8);
    \draw (v0) -- (v6);
    \draw (v1) -- (v8);
    \draw (v1) -- (v6);
    \draw (v2) -- (v8);
    \draw (v2) -- (v4);
    \draw (v3) -- (v8);
    \draw (v3) -- (v5);
    \draw (v4) -- (v7);
    \draw (v5) -- (v7);
    \draw (v7) -- (v8);
  \end{tikzpicture}\\\hrule
  \begin{tikzpicture}[yscale=1.5]
    \anonlist{right}{acdf} (v0) at (.5cm, -.5cm) {};
    \anonlist{left}{abde} (v1) at (-.5cm, .5cm) {};
    \anonlist{right}{abcd} (v2) at (.5cm, .5cm) {};
    \anonlist{right}{abdf} (v3) at (-.5cm, -1.5cm) {};
    \anonlist{left}{acde} (v4) at (-1cm, -1cm) {};
    \anonlist{above}{abce} (v5) at (0cm, 1cm) {};
    \anonlist{left}{abce} (v6) at (-.5cm, -.5cm) {};
    \anonlist{right}{abef} (v7) at (0cm, -1cm) {};
    \anonlist{left}{abcd} (v8) at (0cm, 0cm) {};
    \draw (v0) -- (v8);
    \draw (v0) -- (v7);
    \draw (v1) -- (v8);
    \draw (v1) -- (v5);
    \draw (v2) -- (v8);
    \draw (v2) -- (v5);
    \draw (v3) -- (v4);
    \draw (v3) -- (v7);
    \draw (v4) -- (v6);
    \draw (v6) -- (v8);
    \draw (v6) -- (v7);
  \end{tikzpicture}%\\
  \begin{tikzpicture}
    \anonlist{above}{abde} (v0) at (2cm, 0.0cm) {};
    \anonlist{above}{abce} (v1) at (1cm, .5cm) {};
    \anonlist{above}{abcd} (v2) at (1cm, -.5cm) {};
    \anonlist{right}{abce} (v3) at (0cm, -1cm) {};
    \anonlist{right}{abcf} (v4) at (0cm,  -2cm) {};
    \anonlist{right}{cdef} (v5) at (0cm, -3cm) {};
    \anonlist{left}{adef} (v6) at (-1cm, 0cm) {};
    \anonlist{left}{abcd} (v7) at (-1cm, -1cm) {};
    \anonlist{above}{abcd} (v8) at (0cm, 0cm) {};
    \draw (v0) -- (v1);
    \draw (v0) -- (v2);
    \draw (v1) -- (v8);
    \draw (v2) -- (v8);
    \draw (v3) -- (v6);
    \draw (v3) -- (v7);
    \draw (v4) -- (v6);
    \draw (v4) -- (v7);
    \draw (v5) -- (v6);
    \draw (v5) -- (v7);
    \draw (v6) -- (v8);
    \draw (v7) -- (v8);
  \end{tikzpicture}
  \begin{tikzpicture}
    \anonlist{right}{abce} (v0) at (.5cm, -2cm) {};
    \anonlist{left}{abcf} (v1) at (-.5cm, -2cm) {};
    \anonlist{left}{bcde} (v2) at (-.5cm, 2cm) {};
    \anonlist{right}{abde} (v3) at (.5cm, 2cm) {};
    \anonlist{left}{adef} (v4) at (-.5cm, -1cm) {};
    \anonlist{right}{abcd} (v5) at (.5cm, -1cm) {};
    \anonlist{left}{abce} (v6) at (-.5cm, 1cm) {};
    \anonlist{right}{abcd} (v7) at (.5cm, 1cm) {};
    \anonlist{right}{abcd} (v8) at (0cm, 0cm) {};
    \draw (v0) -- (v4);
    \draw (v0) -- (v5);
    \draw (v1) -- (v4);
    \draw (v1) -- (v5);
    \draw (v2) -- (v6);
    \draw (v2) -- (v7);
    \draw (v3) -- (v6);
    \draw (v3) -- (v7);
    \draw (v4) -- (v8);
    \draw (v5) -- (v8);
    \draw (v6) -- (v8);
    \draw (v7) -- (v8);
  \end{tikzpicture}\\\hrule
  \begin{tikzpicture}
    \anonlist{above}{acfg} (x1) at (.5cm, 0cm) {};
    \anonlist{above}{abde} (x2) at (-.5cm, 0cm) {};
    \anonlist{below}{abce} (y1) at (-2cm, -1cm) {};
    \anonlist{below}{abfg} (y2) at (-1cm, -1cm) {};
    \anonlist{below}{abcd} (y3) at (0cm, -1cm) {};
    \anonlist{below}{adeg} (y4) at (1cm, -1cm) {};
    \anonlist{below}{adef} (y5) at (2cm, -1cm) {};
    \foreach \i in {1,2}
    {
      \foreach \j in {1,...,5}
      {
        \draw (x\i) -- (y\j);
      }
    }
  \end{tikzpicture}%\\
  \begin{tikzpicture}
    \anonlist{above}{abce} (a) at (0cm, 0cm) {};
    \anonlist{above}{abef} (b) at (1cm, 0cm) {};
    \anonlist{above}{acde} (c) at (3cm, 0cm) {};
    \anonlist{below}{acdf} (d) at (0cm, -1cm) {};
    \anonlist{below}{abdf} (e) at (1cm, -1cm) {};
    \anonlist{below}{abcd} (f) at (3cm, -1cm) {};
    \anonlist{below}{abcf} (g) at (2cm, -.5cm) {};
    \draw (a) -- (b) -- (c) -- (f) -- (e) -- (d) -- (a);
    \draw (b) -- (e);
    \draw (b) -- (g) -- (f);
  \end{tikzpicture}%\\
  \begin{tikzpicture}[scale=.7]
    \anonlist{right}{abdf} (y) at (0cm, 0cm) {};
    \foreach \i in {1,...,6}
    {
      \anonlist{above}{} (x\i) at (90 + \i*360/6 : 2cm) {};
    }
    \draw (y) -- (x1);
    \draw (y) -- (x3);
    \draw (y) -- (x5);
    \draw (x1) -- (x2) -- (x3) -- (x4) -- (x5) -- (x6) -- (x1);
    \node[left] at (x1) {abcd};
    \node[left] at (x2) {abce};
    \node[below] at (x3) {adef};
    \node[right] at (x4) {adeg};
    \node[right] at (x5) {abfg};
    \node[above] at (x6) {acfg};
  \end{tikzpicture}
  \label{fig:badgraphs}
  \caption{Some non-$\ab{4}{2}$-choosable graphs.}
\end{figure}
\begin{figure}
\ContinuedFloat
\centering
  \begin{tikzpicture}
    \foreach \i in {1,...,6}
    {
      \anonlist{above}{} (x\i) at (90 + \i*360/6 : 1cm) {};
    }
    \draw (x1) -- (x2) -- (x3) -- (x4) -- (x5) -- (x6) -- (x1);
    \node[left] at (x1) {acde};
    \node[left] at (x2) {abef};
    \node[below] at (x3) {abdf};
    \node[right] at (x4) {cdef};
    \node[right] at (x5) {bcef};
    \node[above] at (x6) {abde};
    \anonlist{left}{abcd} (y2) at (90 + 120 - 10 : 2cm) {};
    \anonlist{above}{abcf} (y6) at (90 + 360 : 2cm) {};
    \draw (x1) -- (y2) -- (x3);
    \draw (x5) -- (y6) -- (x1);
  \end{tikzpicture}%\\
  \begin{tikzpicture}
    \anonlist{above}{abce} (a) at (0cm, 0cm) {};
    \anonlist{above}{abcd} (b) at (1cm, 0cm) {};
    \anonlist{left}{abde} (c) at (0cm, -1cm) {};
    \anonlist{above right}{abcd} (d) at (1cm, -1cm) {};
    \anonlist{right}{acde} (e) at (2cm, -1cm) {};
    \anonlist{below}{adef} (f) at (0cm, -2cm) {};
    \anonlist{below}{abcf} (g) at (1cm, -2cm) {};
    \anonlist{below}{abef} (h) at (2cm, -2cm) {};
    \draw (a) -- (b) -- (d) -- (e) -- (h) -- (g) -- (f) -- (c) -- (a);
    \draw (c) -- (d) -- (g);
  \end{tikzpicture}%\\
  \begin{tikzpicture}
    \anonlist{above}{abde} (a) at (0cm, 0cm) {};
    \anonlist{above}{acef} (b) at (1cm, 0cm) {};
    \anonlist{above}{abcf} (c) at (2cm, 0cm) {};
    \anonlist{above}{abcd} (d) at (3cm, 0cm) {};
    \anonlist{below}{abcd} (e) at (0cm, -1cm) {};
    \anonlist{below}{abce} (f) at (1cm, -1cm) {};
    \anonlist{below}{abef} (g) at (2cm, -1cm) {};
    \anonlist{below}{adef} (h) at (3cm, -1cm) {};
    \draw (a) -- (b) -- (c) -- (d) -- (h) -- (g) -- (f) -- (e) -- (a);
    \draw (b) -- (f);
    \draw (c) -- (g);
  \end{tikzpicture}\\\hrule
  \begin{tikzpicture}
    \foreach \i in {1,...,6}
    {
      \anonlist{above}{} (x\i) at (60 + \i*360/6 : 1cm) {};
    }
    \draw (x1) -- (x2) -- (x3) -- (x4) -- (x5) -- (x6) -- (x1);
    \node[above] at (x1) {acde};
    \node[left] at (x2) {abce};
    \node[below] at (x3) {abcd};
    \node[below] at (x4) {acde};
    \node[right] at (x5) {abce};
    \node[above] at (x6) {abcd};
    \anonlist{left}{abde} (y2) at (60 + 120 : 2.2cm) {};
    \anonlist{right}{abde} (y5) at (60 + 5*60 : 2.2cm) {};
    \draw (x1) -- (y2) -- (x3);
    \draw (x4) -- (y5) -- (x6);
  \end{tikzpicture}
  \caption{Some non-$\ab{4}{2}$-choosable graphs.}
  \label{fig:badgraphs}
\end{figure}
\bibliographystyle{amsplain}\bibliography{chinabib}
\end{document}